\documentclass[review]{siamart220329}

\pdfoutput=1

\usepackage{color}
\usepackage{graphicx,epic} 
\usepackage{amsmath}
\usepackage{amsfonts}
\usepackage{amssymb}
\usepackage{epstopdf}
\usepackage{algorithm}
\usepackage{algpseudocode}

\usepackage{setspace}
\usepackage{xcolor}
\usepackage{lineno} 
\usepackage{url}
\usepackage{hyperref}  

\usepackage[normalem]{ulem}

\newtheorem{remark}{\bf Remark}
\newtheorem{example}{\bf Example}

\newcommand{\mbf}[1]{\mbox{\boldmath$#1$}}

\definecolor{dukeblue}{rgb}{0.0, 0.0, 0.61}
\definecolor{ivogreen}{HTML}{009538}

\begin{document}
	
\nolinenumbers	

\title{Convex optimization problems inspired by geotechnical stability analysis\thanks{Submitted to the editors on January 4, 2025.
\funding{The research was supported by the Ministry of Education, Youth and Sports of the Czech Republic within the OP JAK project INODIN, no. \url{CZ.02.01.01/00/23_020/0008487}, co-funded by EU.}}}

\author{Stanislav Sysala\thanks{Institute of Geonics of the Czech Academy of Sciences, Ostrava, Czech Republic (\email{stanislav.sysala@ugn.cas.cz},
\email{michal.beres@ugn.cas.cz},
\email{simona.beresova@ugn.cas.cz}, \email{tomas.luber@ugn.cas.cz}).}
\and  Michal B\'{e}re\v{s}$^\dagger$\thanks{Faculty of Electrical Engineering and Computer Science, V\v SB - Technical University of Ostrava, Ostrava, Czech Republic.} 
\and Simona B\'{e}re\v{s}ov\'a\footnotemark[2] \footnotemark[3]
\and Jaroslav Haslinger\thanks{Faculty of Mechanical Engineering, V\v SB - Technical University of Ostrava, Ostrava, Czech Republic (\email{hasling@karlin.mff.cuni.cz}).}
\and Jakub Kru\v{z}\'ik\footnotemark[2] \footnotemark[3]
\and Tom\'a\v{s} Luber\footnotemark[2]
}


\maketitle

\begin{abstract}
This paper is motivated by the \textit{limit load}, \textit{limit analysis} and \textit{shear strength reduction} methods, which are commonly employed in geotechnical stability analysis or similar applications. The aim is to make these methods more approachable by introducing a unified framework based on abstract convex optimization and its parametric studies. We establish suitable assumptions on the abstract problems that capture the selected features of these methods and facilitate rigorous theoretical investigation. Further, we propose continuation techniques tailored to the resulting parametric problem formulations and show that the developed abstract framework could also be useful outside the domain of geotechnical stability analysis. The main results are illustrated with analytical and numerical examples. The numerical example deals with a 3D slope stability problem. 

\end{abstract}






\section{Introduction}
\label{sec_intro}

Stability assessment of structures is of utmost importance in civil engineering and geotechnics. It includes the determination of the \textit{factor of safety} (FoS) and the estimation of failure zones for a critical state of loads or other input parameters. The \textit{limit load} (LL), \textit{limit analysis} (LA) and \textit{shear strength reduction} (SSR) methods are frequently used, especially if we build on elasto-plastic (EP) models, mainly perfectly plastic ones, and on the finite element method. 


The LL method is a universal approach where external forces are gradually increased, see, for example, \cite{NPO08}. The corresponding FoS is defined as a critical (limit) value of the load factor. Beyond this value, no solution of the EP problem exists. This incremental technique is closely related to the LA approach where FoS is defined more directly, by solving a convex optimization problem with constraints, see \cite{T85,Ch96,Sl13,KLK16,HRS15}. LA has originally appeared as an analytical approach estimating FoS from above and below, see \cite{C75}. The corresponding lower and upper bounds follow from the duality between stress and kinematic fields within the LA problem, see \cite{T85,Ch96,SHRR21}.

The SSR method \cite{ZHL1975,BB91,DRD99,GL99} is more conventional in slope stability assessment and similar geotechnical applications than the LL method. It was originally proposed for the Mohr-Coulomb EP model and then extended to some other plastic models. The corresponding FoS is defined as a critical value of a scalar factor reducing the strength material parameters of the model. The LA approach can also be used within the SSR method as an alternative definition of FoS, see \cite{TSSLR2015,TSS2015b,KL15,Sysala_2021,STHM23}. 


In this paper, we develop an abstract mathematical background for convex optimization problems in $\mathbb R^n$ inspired by the EP problem and the related LL, LA and SSR methods. Such a treatment is motivated by the following reasons.

First, finite-dimensional algebraic problems are common for researchers in mathematics, engineering or for software developers. Therefore, methods and results presented in this paper should be readable for a broader class of scientists. It is not necessary to know EP models for reading this paper.

Second, we highlight interesting mathematical background of the LL, LA and SSR methods, which is not too known and could be useful in other applications. Due to the abstract setting, we demonstrate that the LA approach can also lead to innovative solvability conditions in convex optimization. Some theoretical results are illustrated with analytical and numerical examples.

Third, we use the abstract framework to derive original relationships among LL, LA and SSR methods and to propose an advanced continuation technique for finding FoS, which is innovative within the SSR method. To make the text easier for the specialists in plasticity and geotechnics, the paper is completed with remarks interpreting the abstract notation.

The most important and difficult task of this paper is to establish convenient assumptions on the abstract analogies of the LL, LA and SSR methods and then establish main features of these methods by tools of mathematical analysis. Our inspiration stems from our earlier research, see \cite{SHHC15, CHKS15, HRS15, HRS16, RSH18, HRS19,SHRR21} for the LL and LA methods and \cite{Sysala_2021, STHM23} for the SSR methods. This topic is also closely related to variational problems with linear growth functionals having various applications, see, for example, \cite{R89,HZ94,CLR06}.

The rest of the paper is organized as follows. In Section \ref{sec_algebra}, we introduce an abstract algebraic counterpart of the EP problem, specify key assumptions and prove useful auxiliary results. Sections \ref{sec_LA} and \ref{subsec_LA} are devoted to the algebraic analogies of the LL and LA methods, respectively. We highlight the importance of the LA problem on solvability in convex optimization. In Section \ref{sec_SSRM}, we analyze an algebraic problem inspired by the SSR method and relate it with the LA approach. In Section \ref{sec_num_example}, the Mohr-Coulomb EP problem is briefly recapitulated and a numerical example on slope stability in 3D is presented and solved by the LL and SSR methods. Concluding remarks are given in Section \ref{sec_conclusion}.


\section{Abstract algebraic problem and key assumptions}
\label{sec_algebra}

Consider the following algebraic system of nonlinear equations in $\mathbb R^n$:
\begin{equation}
	\mbox{find } u\in\mathbb R^n: \quad F(u) = b, \qquad  F\colon\mathbb R^n\rightarrow\mathbb R^n, \; b\in\mathbb R^n,
	\label{alg_system}
\end{equation}
with the following key assumptions, which are considered throughout the paper. Other (more specific) assumptions will be introduced in the sequel.

\begin{itemize}
\item[$(\mathcal A_1)$] \textit{$b\neq0$, $F(0)=0$ and $F$ is continuous in $\mathbb R^n$.}
\item[$(\mathcal A_2)$] \textit{$F$ has a convex potential in $\mathbb R^n$}, i.e., there exists a convex and continuously differentiable function $\mathcal I\colon\mathbb R^n\rightarrow \mathbb R$ such that $\nabla\mathcal I(v)=F(v)$ for any $v\in\mathbb R^n$. Without loss of generality, we additionally assume that $\mathcal I(0)=0$.
\item[$(\mathcal A_3)$] \textit{$\mathcal I$ has at least linear growth at infinity}, i.e., 
\begin{equation}
\text{there exist}\;  c_1>0,\;c_2\geq0:\quad \mathcal I(v)\geq c_1\|v\|-c_2\quad\text{for  all }\, v\in\mathbb R^n.
\label{lin_growth}
\end{equation}
\end{itemize}

\begin{remark}
(EP interpretation.) \emph{In EP, the system \eqref{alg_system} arises from a discrete setting of the problem in terms of displacements, see Section \ref{sec_num_example}. $F$ is usually Lipschitz continuous and even semismooth in EP. The assumption $(\mathcal A_2)$ holds for associated EP models obeying the principle of maximum plastic dissipation, or for nonassociated models, which are approximated by a series of associated models (e.g. Davis' approximation). The assumption $(\mathcal A_3)$ is usual for elastic-perfectly plastic models or for EP models with bounded hardening. For EP models with unbounded hardening, a quadratic growth at infinity is expected leading to much stronger solvability results than the presented ones.}
\end{remark}

From the mathematical point of view, the assumption $(\mathcal A_2)$ enables us to analyze the system \eqref{alg_system} by the following minimization problem:
	\begin{equation}
\mbox{find } u\in\mathbb R^n:\quad\mathcal J(u)\leq \mathcal J(v)\;\;\text{for all }\, v\in\mathbb R^n,\quad\mbox{where}\quad\mathcal J(v)=\mathcal I(v)-b^\top v.
\label{J_min}
	\end{equation}
It is well-known that $u\in\mathbb R^n$ solves \eqref{alg_system} if and only if it is a minimum of $\mathcal J$ in $\mathbb R^n$. Therefore, the problems \eqref{alg_system} and \eqref{J_min} have the same solution set, which will be denoted by $\mathcal K$ from now on.

In what follows, we shall present basic properties of $\mathcal K$. First of all, the set $\mathcal K$ is closed and convex due to the convexity and continuity of $\mathcal J$. Next results are formulated under the assumptions $(\mathcal A_1)-(\mathcal A_3)$, despite the fact that they can be proven for more general convex functions $\mathcal J$.

\begin{lemma}
	Let $(\mathcal A_1)-(\mathcal A_3)$ be satisfied and the function $\mathcal J$ be defined by \eqref{J_min}.
	Then the following statements are equivalent:
	\begin{itemize}
		\item[(a)] $\mathcal J$ is coercive in $\mathbb R^n$;
		\item[(b)] The solution set $\mathcal K$ is nonempty and bounded;
		\item[(c)] There exist constants $\tilde c_{1}>0$ and $\tilde c_{2}\geq0$ such that 
		\begin{equation}
			\mathcal J(v)\geq \tilde c_{1}\|v\|-\tilde c_{2}\quad\text{for all } v\in\mathbb R^n.
			\label{lin_growth3}
		\end{equation}
	\end{itemize}
\label{lemma_K_t}
\end{lemma}

\begin{proof}
$(a)\Rightarrow(b)$. If $\mathcal J$ is continuous and coercive in $\mathbb R^n$ then it is well-known that the solution set $\mathcal K$ to \eqref{J_min} is nonempty and bounded, see, for example, \cite{B99}.

$(b)\Rightarrow(c)$. Let $\mathcal K$ be nonempty and bounded. Then there exists $\rho>0$  such that $\|u\|<\rho$ for any $u\in\mathcal K$. Since $\partial B_\rho:=\{w\in\mathbb R^n\ |\; \|w\|=\rho\}$ is compact, the function $\mathcal J$ attains a minimum in $\partial B_\rho$. Hence, there exists $c>0$ such that
\begin{equation}
\mathcal J(w)\geq q+c\quad \text{for all } w\in\partial B_\rho,\quad\mbox{where } q=\inf_{v\in\mathbb R^n}\mathcal J(v)=\mathcal J(u).
\label{bound_compact}
\end{equation}
Let $v\in\mathbb R^n$, $\|v\|>\rho$. Then there exists $\beta_v\in(0,1)$ such that $\|(1-\beta_v)u+\beta_v v\|=\rho$. From the triangular inequality, it easily follows:
\begin{equation}
    \rho\geq\beta_v\|v\|-(1-\beta_v)\|u\|,\quad\mbox{i.e., } \;\;\beta_v\leq\frac{\rho+\|u\|}{\|v\|-\|u\|}.
    \label{bound_b}
\end{equation}
From convexity of $\mathcal J$ and \eqref{bound_compact}, we obtain:
\begin{equation*}
(1-\beta_v)q+\beta_v\mathcal J(v)= (1-\beta_v)\mathcal J(u)+\beta_v\mathcal J(v)\geq\mathcal J((1-\beta_v)u+\beta_v v)\stackrel{\eqref{bound_compact}}{\geq} q+c,
\end{equation*}
so that 
\begin{equation*}
\mathcal J(v)\geq \frac{c}{\beta_v}+ q\stackrel{\eqref{bound_b}}{\geq}\frac{c(\|v\|-\|u\|)}{\rho+\|u\|}+ q\geq \frac{c(\|v\|-\rho)}{2\rho}+ q.
\label{convexity_bound}
\end{equation*}
Therefore, \eqref{lin_growth3} holds for any $v\in\mathbb R^n$, $\|v\|>\rho$.
Let $v\in\mathbb R^n$, $\|v\|\leq\rho$. Then 
$\mathcal J(v)\geq q\geq \|v\|+q-\rho$. Thus \eqref{lin_growth3} holds for any $v\in\mathbb R^n$.

$(c)\Rightarrow(a)$. This statement is obvious.
\end{proof}

\begin{lemma}
	Let $(\mathcal A_1)-(\mathcal A_3)$ and the condition $\|b\|<c_1$ be satisfied. Then $\mathcal K$ is nonempty and bounded.
	\label{lemma_solv}
\end{lemma}

\begin{proof}
	From \eqref{lin_growth} and the definition of $\mathcal J$, we obtain the coercivity of $\mathcal J$:
	\begin{equation}
		\mathcal J(v)=\mathcal I(v)-b^\top v\geq (c_1-\|b\|)\|v\|-c_2\rightarrow+\infty\quad\mbox{as}\quad \|v\|\rightarrow+\infty.
		\label{J_bound}
	\end{equation}
	By Lemma \ref{lemma_K_t}, the set $\mathcal K$ is nonempty and bounded.
\end{proof}
From Lemma \ref{lemma_solv}, one can deduce that solvability of \eqref{alg_system} and \eqref{J_min} is not guaranteed for larger norms of the vector $b$. In Sections \ref{sec_LA} and \ref{sec_SSRM}, we present two different parametrizations of \eqref{alg_system} enabling us to study solvability of this system and related \textit{factors of safety} (FoS). Let us note that a more advanced result dealing with solvability of \eqref{J_min} is presented in Section \ref{subsec_solvability}.


\section{Parametrization of the vector $b$}
\label{sec_LA}

In this section, we consider the following parametrization of the system \eqref{alg_system}:
\begin{equation}
(\mathcal P_t)\quad \mbox{given $t\geq 0,\;$ find } u_t\in\mathbb R^n: \quad F(u_t)=tb.
\label{system_t}
\end{equation}
We analyze it under the assumptions $(\mathcal A_1)-(\mathcal A_3)$ and propose an advanced continuation technique describing the solution path. Analogously to Section \ref{sec_algebra}, we introduce the minimization form of $(\mathcal P_t)$:
\begin{equation}
	\mbox{find } u_t\in\mathbb R^n:\quad\mathcal J_t(u_t)\leq \mathcal J_t(v)\;\;\text{for all } v\in\mathbb R^n,\quad\mbox{where}\quad\mathcal J_t(v)=\mathcal I(v)-tb^\top v,
	\label{J_t_min}
\end{equation}
and denote the solution set as $\mathcal K_t$. Next, we define the limit value of $t$ representing FoS:
\begin{equation}
	t^*\;:=\;\mbox{{supremum} of } \; \bar t\geq 0\;\;\mbox{such that $\mathcal K_t$ is nonempty for any $t\in[0,\bar t\,]$.}
	\label{limit_load}
\end{equation}
By $(\mathcal A_1)$, the solution of $(\mathcal P_t)$ exists at least for $t=0$ and, thus, the definition of $t^*$ is meaningful. The case $t^*=+\infty$ may also occur, for example if $\mathcal I$ is quadratic with positive definite Hessian. The system $(\mathcal P_t)$ need not to have a solution for $t=t^*<+\infty$. If $t^*<1$ then the original system \eqref{alg_system} has no solution.

We start with the following fundamental result.
\begin{theorem}
Let $(\mathcal A_1)-(\mathcal A_3)$ be satisfied. Then $t^*>0$, $\mathcal K_t$ is nonempty and bounded for any $t\in[0,t^*)$ and $\mathcal K_{t^*}$ is either unbounded or empty for $t^*<+\infty$. In addition,
\begin{equation}
t^*=\sup\{t\geq0\ |\; \inf_{v\in\mathbb R^n}\mathcal J_t(v)>-\infty\}.
\label{limit_load2}
\end{equation}
\label{theorem_K}
\end{theorem}

\begin{proof}
By Lemma \ref{lemma_solv}, $(\mathcal P_t)$ has a solution for sufficiently small $t>0$ satisfying $c_1>t\|b\|$, where $c_1>0$ is as in \eqref{lin_growth}. Therefore, $t^*>0$.

Let $\bar t< t^*$ be such that $\mathcal K_{\bar t}\neq\emptyset$ and $t\in[0,\bar t)$. Then $\inf_{v\in\mathbb R^n}\mathcal J_{\bar t}(v)=c>-\infty$ and
$$\mathcal J_t(v)=\frac{\bar t- t}{\bar t}\mathcal I(v)+\frac{t}{\bar t}\mathcal J_{\bar t}(v)\geq \frac{\bar t- t}{\bar t}(c_1\|v\|-c_2)+\frac{t}{\bar t}c\quad\text{for all } v\in\mathbb R^n.$$
From this and Lemma \ref{lemma_K_t}, it follows that $\mathcal K_t$ is nonempty and bounded. Since $\bar t< t^*$ was arbitrarily chosen, $\mathcal K_t$ is nonempty but bounded for any $t\in[0,t^*)$. In addition, \eqref{limit_load2} holds and $\mathcal K_t$ can be unbounded only for $t=t^*$.

Let $t^*<+\infty$ and suppose that $\mathcal K_{t^*}$ is nonempty and bounded. By Lemma \ref{lemma_K_t}, there exist $c_1^*>0$ and $c_2^*\geq0$ such that 
$$\mathcal J_{t^*}(v)\geq c_1^*\|v\|-c_2^*\quad\text{for all } v\in\mathbb R^n.$$
Hence, for any $\epsilon\in(0,c_1^*/\|b\|)$, we have
$$\mathcal J_{t^*+\epsilon}(v)=\mathcal J_{t^*}(v)-\epsilon b^\top v\geq (c_1^*-\epsilon\|b\|)\|v\|-c_2^*\quad\text{for all } v\in\mathbb R^n.$$
Using Lemma \ref{lemma_K_t} once again, we see that $\mathcal K_t\neq\emptyset$ also for $t=t^*+\epsilon$, which contradicts the definition of $t^*$. Therefore, $\mathcal K_{t^*}$ is either empty or unbounded.
\end{proof}

\begin{remark}
	(EP interpretation.) \emph{In EP, the safety factor $t^*$ represents the \textit{limit load}. For $t\in[0,t^*)$, the unique solution $u_t\in\mathcal K_t$ is expected on the basis of numerical experiments with the Newton method, because it is observed that generalized Hessians to $\mathcal J_t$ are positive definite in vicinity of the solution.}
\end{remark}

Now, we focus on the determination of $t^*$. The simplest way how to find $t^*$ is to gradually increase the parameter $t$. However, such a treatment has several drawbacks. First, it is not guaranteed that the problem $(\mathcal P_t)$ has a solution for a chosen value of $t$. From the numerical point of view, it is difficult to distinguish whether the problem does not have any solution or an iterative solver (e.g. Newton-like method) converges too slowly. Secondly, a small change of $t$ may cause a very large change of $u_t$ in the vicinity of $t^*$. Therefore, our aim is to introduce a more suitable continuation technique, which is based on the following results.

\begin{lemma}
	Let $(\mathcal A_1)-(\mathcal A_3)$ be satisfied, $0\leq t_1<t_2\leq t^*$, $u_1\in \mathcal K_{t_1}$ and $u_2\in \mathcal K_{t_2}$. Then
	\begin{equation}
		b^\top u_1<b^\top u_2.	
		\label{bu}
	\end{equation}
	\label{lemma_bu}
\end{lemma}

\begin{proof}
	It holds:
	\begin{eqnarray*}
		\mathcal J_{t_2}(u_2)&\leq& \mathcal J_{t_2}(u_1)=\mathcal J_{t_1}(u_1)+(t_1-t_2)b^\top u_1\leq \mathcal J_{t_1}(u_2)+(t_1-t_2)b^\top u_1\\
		&=&\mathcal J_{t_2}(u_2)+(t_1-t_2)b^\top (u_1-u_2).
	\end{eqnarray*}
	Hence, $b^\top u_1\leq b^\top u_2.$ Suppose that $b^\top u_1= b^\top u_2.$ Then the equality must hold in the above chain of inequalities, which means that $u_1$ solves $(\mathcal P_{t_1})$ and also $(\mathcal P_{t_2})$, i.e., $F(u_1)=t_1 b=t_2 b$. This contradicts the assumption $(\mathcal A_1)$ saying that $b\neq0$. Therefore, \eqref{bu} is satisfied.
\end{proof}

\begin{theorem}
Let $(\mathcal A_1)-(\mathcal A_3)$ be satisfied. Then for any $\omega\geq0$ there exists a unique $t_\omega\in[0,t^*]$ for which the problem $(\mathcal P_{t_\omega})$ has a solution $u_\omega\in\mathcal K_{t_\omega}$ such that $b^\top u_\omega=\omega$. In addition, the function $\psi\colon\omega\mapsto t_\omega$ is non-decreasing, continuous in $\mathbb R_+$, $\psi(0)=0$, and $\lim_{\omega\rightarrow+\infty}\psi(\omega)=t^*$.
\label{theorem_omega}
\end{theorem}

\begin{proof}
Let $\omega\geq0$ be given. Since $b\neq0$, from $(\mathcal A_1)-(\mathcal A_3)$ it follows that there exists $u_\omega\in\mathbb R^n$, $b^\top u_\omega=\omega$, such that 
\begin{equation}
\mathcal I(u_\omega)=\min_{\substack{v\in\mathbb R^n\\ b^\top v=\omega}} \mathcal I(v).
\label{I_omega_min}
\end{equation}
In addition, $u_\omega$ solves \eqref{I_omega_min} if and only if the following optimality condition is satisfied:
\begin{equation}
    F(u_\omega)^\top v=0\quad\text{for all } v\in\mathbb R^n,\;\mbox{such that }b^\top v=0,
\end{equation}
i.e., if and only if the vectors $F(u_\omega)$ and $b$ are parallel. Hence, there exists $t_\omega\in\mathbb R$ such that
\begin{equation}
F(u_\omega)=t_\omega b;\quad b^\top u_\omega=\omega.
\label{saddle-system}
\end{equation}
From \eqref{saddle-system}, convexity and differentiability of $\mathcal I$, we obtain:
$$t_\omega\omega=t_\omega b^\top u_\omega=F(u_\omega)^\top u_\omega=[F(u_\omega)-F(0)]^\top (u_\omega-0)\geq 0,$$
and so $t_\omega\geq0$. Further, we see that $u_\omega\in\mathcal K_{t_\omega}$. Therefore, $t_\omega\leq t^*$ as follows from the definition of $t^*$. If there exist $t_1,t_2\in[0,t^*]$ and $u_1\in\mathcal K_{t_1}$, $u_2\in\mathcal K_{t_2}$ such that $b^\top u_1=b^\top u_2=\omega$, then $t_1=t_2$ by Lemma \ref{lemma_bu}. Therefore, $t_\omega$ is uniquely defined.

The function $\psi\colon\omega\mapsto t_\omega$ is non-decreasing as follows from Lemma \ref{lemma_bu}. To prove continuity of $\psi$, consider a sequence $\{\omega_n\}_n\subset(\omega-\delta,\omega+\delta)$, $\delta>0$, converging to $\omega>0$. Let $(u_n,t_n)$ denote a solution to \eqref{saddle-system} for $\omega:=\omega_n$. Since $\psi$ is non-decreasing, we have $t_n\leq t_{\omega+\delta}$. Next,
\begin{align*}
c_1\|u_n\|-c_2&\leq \mathcal I(u_n)=\mathcal J_{t_n}(u_n)+t_nb^\top u_n\leq \mathcal J_{t_n}(u_{\omega+\delta})+t_nb^\top u_n\\
&=\mathcal I(u_{\omega+\delta})+t_n(b^\top u_n-b^\top u_{\omega+\delta})\leq \mathcal I(u_{\omega+\delta})\quad \text{for all } n>0.
\end{align*}
Hence, the sequences $\{u_n\}$ and $\{t_n\}$ are bounded and it is easy to see that their accumulation points solve  \eqref{saddle-system}. Since the solution component $t_{\omega}$ is unique,  $t_n\rightarrow t_\omega$ as $n\rightarrow+\infty$. Therefore, the function $\psi\colon\omega\mapsto t_\omega$ is continuous.

From the established properties of $\psi$, it follows that there exists the limit of $\psi(\omega)$ for $\omega\rightarrow+\infty$, which is less than or equal to $t^*$. Suppose that this limit is less than $t \in (0, t^*)$. By the definition of $t^*$, the solution set $\mathcal K_t$ is nonempty and thus there exists $u_t\in\mathcal K_t$. If we set $\tilde\omega=b^\top u_t$ then $\psi(\tilde\omega)=t$ which is a contradiction to $\lim_{\omega\rightarrow+\infty}\psi(\omega)<t$. Therefore, $\lim_{\omega\rightarrow+\infty}\psi(\omega)=t^*$.
\end{proof}

On the basis of Theorem \ref{theorem_omega}, one can use the function $\psi:\omega\mapsto t_\omega$ to suggest more advanced continuation method, see Figure \ref{fig_omega}. In particular, $\omega$ is a control parameter, which is increased up to $+\infty$. For any $\omega$, the extended system \eqref{saddle-system} is solved. The existence of its  solution $(u_\omega,t_\omega)$ is the main advantage of this \textit{indirect continuation} in comparison with the direct continuation using $t$. One can start with a constant increment of $\omega$. If it is observed that the corresponding increments of $t_\omega$ are negligible then it is convenient to increase the increments of $\omega$ to be sure that we are close to the searched value $t^*$.

\begin{figure}
	\begin{center}
		\begin{picture}(140,90)
			\put(0,0){\vector(0,1){90}} \put(0,0){\vector(1,0){140}}
			{\thicklines
				\put(0,0){\line(1,2){20}}
				\qbezier(20,40)(30,70)(140,70)
			}
			\multiput(0,73)(4,0){35}{\put(0,0){\line(1,0){2}}}
			
			\put(-4,88){\makebox(0,0)[r]{$t$}}
			\put(138,67){\makebox(0,0)[t]{$\psi(\omega)$}}
			\put(-2,72){\makebox(0,0)[r]{$t^*$}}
			\put(138,-4){\makebox(0,0)[t]{$\omega$}}
		\end{picture}
	\end{center}
	\caption{Scheme of the indirect continuation method with the control parameter $\omega$.}	
	\label{fig_omega}
\end{figure}

\begin{remark}(EP interpretation.)
	\emph{In EP, the term $b^\top u$ represents the work of external forces. Therefore, the control parameter $\omega$ has a clear physical meaning. The indirect continuation method was originally introduced in \cite{SHHC15, CHKS15} for the von Mises yield criterion and then extended in \cite{HRS15, HRS16, SHRR21} to more general EP models and even to infinite dimensional EP problems. Unlike these papers, we have proven the properties of the function $\psi$ without necessity to define a dual problem to $(\mathcal P_t)$.     
    }
\end{remark}


\section{Limit analysis problem and its importance}
\label{subsec_LA}

In this section we present a specific convex optimization problem, called \textit{the limit analysis (LA) problem}, which enables us to define and analyze the safety factor $t^*$ in another way, see Section \ref{subsec_LA_derivation}. We also illustrate how the achieved theoretical results can be applied to examples with known analytical solutions. In Section \ref{subsec_solvability}, we show that the LA problem has also a potential to be useful in solvability analysis of more general convex optimization problems.

\subsection{Formulation and analysis of the LA problem}
\label{subsec_LA_derivation}

Let us recall that there exists $(u_\omega,t_\omega)$ such that \eqref{saddle-system} is satisfied for a given $\omega>0$. If we set $w_\omega=u_\omega/\omega$, then the pair $(w_\omega,t_\omega)$ solves 
\begin{equation}
	F(\omega w_\omega)=t_\omega b;\quad b^\top w_\omega=1.
	\label{saddle-system2}
\end{equation}
This system represents the KKT conditions characterizing a saddle-point of the Lagrangian
$$\mathcal L(v,t):=\mathcal I_\omega(v)-t(b^\top v-1),$$
where
\begin{equation}
\mathcal I_\omega(v)=\frac{1}{\omega}\mathcal I(\omega v),\quad \nabla\mathcal I_\omega(v)=F(\omega v).
	\label{I_omega}
\end{equation}
Let us observe that $t_\omega$ is the Lagrange multiplier releasing the constraint $b^\top v_\omega=1$. From the existence of the saddle-point $(w_\omega,t_\omega)$, we have the following duality relationships, see \cite{ET74}:
\begin{eqnarray}
\tilde t_\omega& :=&\mathcal I_\omega(w_\omega)=\min_{\substack{v\in\mathbb R^n\\ b^\top v=1}} \mathcal I_\omega(v)\label{duality0}\\
&=&\min_{v\in \mathbb R^n}\sup_{t\in\mathbb R}\mathcal L(v,t)=\max_{t\in\mathbb R}\inf_{v\in \mathbb R^n}\mathcal L(v,t)=\nonumber\\ 
&=&\max_{t\in\mathbb R}\Big\{t+\inf_{v\in\mathbb R^n}[\mathcal I_\omega(v)-tb^\top v]\Big\}=\max_{t\in\mathbb R}\Big\{t+\frac{1}{\omega}\inf_{v\in\mathbb R^n}[\mathcal I(v)-tb^\top v]\Big\}\nonumber\\
&=&\max_{t\in[0,t^*]}\Big\{t+\frac{1}{\omega}\mathcal J_t(u_t)\Big\}=t_\omega+\frac{1}{\omega}\mathcal J_{t_\omega}(u_{\omega}).
\label{duality}
\end{eqnarray}
Since $\mathcal J_{t_\omega}(u_{\omega})\leq\mathcal J_{t_\omega}(0)=0$, (\ref{duality0}), (\ref{duality}) and Theorem \ref{theorem_omega} yield
\begin{equation}
\tilde t_\omega\leq t_\omega\leq t^*\quad\text{for all } \omega>0,
\label{tilde_t_bound}
\end{equation}
\begin{equation}
	\tilde t_\omega\geq t+\frac{1}{\omega}\mathcal J_t(u_t)\stackrel{(\omega\rightarrow+\infty)}{\rightarrow} t\quad\text{for all } t\in[0,t^*),
	\label{tilde_t_bound2}
\end{equation}
and consequently,
\begin{equation}
	\lim_{\omega\rightarrow+\infty}\tilde t_\omega=\lim_{\omega\rightarrow+\infty} t_\omega=t^*.
	\label{tilde_t_bound3}
\end{equation}
Next, we investigate the properties of the convex functions $\{\mathcal I_\omega\}_\omega$ for $\omega\rightarrow+\infty$. We have the following result.
\begin{lemma}
	Let the assumptions $(\mathcal A_1)-(\mathcal A_3)$ be satisfied. Then for any $v\in\mathbb R^n$ the mapping $\omega\mapsto\mathcal I_\omega(v)$ is nondecreasing in $\mathbb R_+$ and 
	\begin{equation}
		\mathcal I_{\omega}(v)\leq\mathcal I_\infty(v)\quad\text{for all } \omega>0,
		\label{I_omega+}
	\end{equation}
	where
	\begin{equation}
		\mathcal I_\infty\colon \mathbb R^n\rightarrow \mathbb R_+\cup\{+\infty\},\quad \mathcal I_\infty(v):=\lim_{\omega\rightarrow+\infty}\mathcal I_{\omega}(v),\quad v\in\mathbb R^n.
		\label{I_infty}
	\end{equation}
In addition, the effective domain of $\mathcal I_\infty$,
\begin{equation}
	\mathcal C:=\mathrm{dom}\, \mathcal I_\infty=\{v\in\mathbb R^n\ |\;\; \mathcal I_\infty(v)<+\infty\},
	\label{C}
\end{equation}
is a nonempty convex cone with vertex at $v=0$ and the function $\mathcal I_\infty$ is convex, positively homogeneous of degree 1 and coercive in $\mathcal C$. If $S$ is an open convex subset of $\mathcal C$ then the restriction of $\mathcal I_\infty$ on $S$ is a continuous function in $S$.
\label{lemma_I_infty}
\end{lemma} 
\begin{proof}
	Let $0<\omega_1\leq\omega_2$ and $v\in\mathbb R^n$ be given. Then from the assumptions on $\mathcal I$, we have:
	$$\mathcal I_{\omega_1}(v)=\frac{1}{\omega_1}\mathcal I(\omega_1v)=\frac{1}{\omega_1}\mathcal I\Big(\frac{\omega_1}{\omega_2}\omega_2v\Big)\leq\frac{1}{\omega_2}\mathcal I(\omega_2v)+\frac{1}{\omega_1}\Big(1-\frac{\omega_1}{\omega_2}\Big)\mathcal I(0)=\mathcal I_{\omega_2}(v).$$ 
	This inequality proves \eqref{I_omega+}.
	
	Clearly, $\mathcal I_\infty(0)=0$. Therefore the set $\mathcal C$ is nonempty. Let $\alpha>0$. Then 
	$$\mathcal I_\infty(\alpha v)=\lim_{\omega\rightarrow+\infty}\frac{1}{\omega}\mathcal I(\omega\alpha v)=\alpha\lim_{\omega\rightarrow+\infty}\frac{1}{\alpha\omega}\mathcal I(\omega\alpha v)=\alpha\mathcal I_\infty(v),$$
	i.e., $\mathcal I_\infty$ is positively homogeneous of degree 1 (1-positively homogeneous). Consequently, $\mathcal C$ is a cone.
	
	Let $u,v\in\mathcal C$ and $\alpha\in[0,1]$. From convexity of $\mathcal I_{\omega}$, it follows
	$$\mathcal I_\infty(\alpha u+(1-\alpha)v)\leq \alpha\mathcal I_\infty(u)+(1-\alpha)\mathcal I_\infty(v)<+\infty.$$
	Hence, $\mathcal C$ is a convex set and $\mathcal I_\infty$ is a convex function in $\mathcal C$.
	
	From \eqref{lin_growth} and the definition of $\mathcal I_\infty$, we have
	$$\mathcal I_\omega(v)\geq c_1\|v\|-\frac{c_2}{\omega},\quad\mbox{so that } \;\mathcal I_\infty(v)\geq c_1\|v\|.$$
	This inequality says that $\mathcal I_\infty$ is coercive in $\mathcal C$. Finally, continuity of $\mathcal I_\infty$ on any open convex subset of $\mathcal C$ follows from \cite[Theorem 10.1]{R72}. 
\end{proof}

The \textit{limit analysis} (LA) problem is the following constrained minimization problem defining a value $t_\infty\in\mathbb R_+\cup\{+\infty\}$:
\begin{equation}
t_\infty:=\inf_{\substack{v\in\mathbb R^n\\ b^\top v=1}}\ \mathcal I_\infty(v)=\inf_{\substack{v\in\mathcal C\\ b^\top v=1}}\ \mathcal I_\infty(v)
\label{LA_problem}
\end{equation}
If the set $\{v\in\mathcal C\ |\; b^\top v=1\}$ is empty then $t_\infty=+\infty$ as it is usual in convex analysis. Otherwise, $t_\infty$ is a finite and nonnegative value. To analyze the LA problem and relate $t_\infty$ to $t^*$, we need the additional assumptions which complete $(\mathcal A_1)-(\mathcal A_3)$:
\begin{itemize}
	\item[$(\mathcal A_4)$] \textit{For any sequence $\{v_\omega\}\subset\mathbb R^n$ such that $\{\mathcal I_\omega(v_\omega)\}$ is bounded and $v_\omega\rightarrow v_\infty$ as $\omega\rightarrow+\infty$, the following properties hold:} 
		\begin{equation}
		v_\infty \in\mathcal C\quad\mbox{and}\quad \liminf_{\omega\rightarrow+\infty}\mathcal I_\omega(v_\omega)\geq\mathcal I_\infty(v_\infty).
		\label{A4}
		\end{equation}
	\item[$(\mathcal A_5)$] \textit{For any $v\in\mathcal C$ there exists a constant $c_v>0$ such that}
	\begin{equation}
		\omega\mathcal I_\infty(v)-c_v\leq \mathcal I(\omega v)\leq \omega\mathcal I_\infty(v)\quad\mbox{for all } \omega\geq0.
		\label{I_infty_bound}
	\end{equation}
\end{itemize}

\begin{theorem}
Let the assumptions $(\mathcal A_1)-(\mathcal A_4)$ be satisfied. Then
\begin{equation}
	t^*=t_\infty.
	\label{LA}
\end{equation}
In addition, if $t^*<+\infty$ then the sequence $\{w_\omega\}_\omega$ of solutions to \eqref{saddle-system2} is bounded and its accumulation points solve \eqref{LA_problem}.
\label{theorem_LA}
\end{theorem}

\begin{proof}
We have:
\begin{equation}
t^*\stackrel{\eqref{tilde_t_bound3}}{=}\lim_{\omega\rightarrow+\infty}\tilde t_\omega\stackrel{\eqref{duality0}}{=}\lim_{\omega\rightarrow+\infty}\mathcal I_\omega(w_\omega)\stackrel{\eqref{duality0}}{=}\lim_{\omega\rightarrow+\infty}\min_{\substack{v\in\mathbb R^n\\ b^\top v=1}} \mathcal I_\omega(v)\stackrel{\eqref{I_omega+}}{\leq}\inf_{\substack{v\in\mathcal C\\ b^\top v=1}}\ \mathcal I_\infty(v)\stackrel{\eqref{LA_problem}}{=}t_\infty.
\label{tilde_t_bound4}
\end{equation}	
If $t^*=+\infty$ then also $t_\infty=+\infty$. Now, assume that $t^*<+\infty$. For any $\omega>0$,
$$c_1\|w_\omega\|-\frac{c_2}{\omega}\stackrel{\eqref{lin_growth}}\leq\frac{1}{\omega}\mathcal I(\omega w_\omega)=\mathcal I_{\omega}(w_\omega)=\tilde t_{\omega}\stackrel{\eqref{tilde_t_bound}}{\leq} t^*,$$
thus, the sequences $\{w_\omega\}_\omega$ and $\{\mathcal I_{\omega}(v_\omega)\}$ are bounded. Then there exists a subsequence $\{w_{\omega'}\}_{\omega'}$ of $\{w_\omega\}_\omega$ converging to some $w_\infty\in\mathbb R^n$. Clearly, $b^\top w_\infty=1$. From $(\mathcal A_4)$, we have $w_\infty\in\mathcal C$ and $\liminf_{\omega\rightarrow+\infty}\mathcal I_{\omega'}(w_{\omega'})\geq\mathcal I_\infty(w_\infty)$.
Hence,
$$t_\infty\leq\mathcal I_{\infty}(w_{\infty})\leq\liminf_{\omega'\rightarrow+\infty} \mathcal I_{\omega'}(w_{\omega'})\stackrel{\eqref{tilde_t_bound4}}{=}t^*\stackrel{\eqref{tilde_t_bound4}}{\leq}t_\infty<+\infty,$$
and consequently $t^*=t_\infty$. In addition, $w_\infty$ is a minimizer in \eqref{LA_problem}. 
\end{proof}

\begin{theorem}
Let the assumptions $(\mathcal A_1)-(\mathcal A_5)$ be satisfied and $t^*<+\infty$. Then the function $\mathcal J_{t^*}$ is bounded from below.
\label{theorem_T*_bound}
\end{theorem}

\begin{proof}
Suppose that $\mathcal J_{t^*}$ is unbounded from below. Since this function is convex, there exists $v_0\in\mathbb R^n$ such that
\begin{equation}
\lim_{n\rightarrow+\infty}\mathcal J_{t^*}(nv_0)=\lim_{n\rightarrow+\infty}[\mathcal I(nv_0)-t^*b^\top(nv_0)]=-\infty.
\label{contradiction}
\end{equation}
This is possible only if $b^\top v_0>0$. Without loss of generality, one can assume that $b^\top v_0=1$. Then
$$\lim_{n\rightarrow+\infty}n\left[\frac{1}{n}\mathcal I(nv_0)-t^*\right]=-\infty.$$
From \eqref{I_infty}, we obtain $\mathcal I_{\infty}(v_0)\leq t^*$, i.e., $v_0\in\mathcal C$. By \eqref{LA_problem} and \eqref{LA}, we have $\mathcal I_{\infty}(v_0)\geq t_\infty= t^*$, i.e.,  $\mathcal I_{\infty}(v_0)= t^*$. Hence,
$$\mathcal J_{t^*}(nv_0)=\mathcal I(nv_0)-nt^*b^\top v_0\stackrel{(\mathcal A_5)}{\geq}n\mathcal I_\infty(v_0)-c_{v_0}-nt^*=-c_{v_0},$$
for any $n\in\mathbb N$, which is a contradiction with \eqref{contradiction}.
\end{proof}

\begin{remark}
\emph{If $(\mathcal A_1)-(\mathcal A_5)$ are satisfied then it is also possible to show that the set $\mathcal C=\{v\in\mathbb R^n\ |\;\; \mathcal I_\infty(v)<+\infty\}$ is closed. Then, continuity of $\mathcal I_\infty$ can be extended to the whole cone $\mathcal C$. In addition, $t^*<+\infty$ if and only if the set $\{v\in\mathcal C\ |\; b^\top v=1\}$ is nonempty. For the sake of brevity, we skip a proof of these statements.}
\end{remark}

\begin{remark}(EP interpretation.)
\emph{In EP, the function $\mathcal I_\infty$ represents the plastic dissipation, which is independent of the elastic part of the model. The shape of the set $\mathcal C$ depends mainly on a chosen plastic criterion. For the von Mises and Tresca criteria, $\mathcal C$ is a subspace of $\mathbb R^n$. For the Drucker-Prager and Mohr-Coulomb criteria, $\mathcal C$ is a  nontrivial cone. The case $\mathcal C=\mathbb R^n$ occurs for criteria leading to bounded yield surfaces. The LA problem \eqref{LA_problem} corresponds to the so-called \textit{upper bound limit analysis theorem}, see \cite{T85, Ch96, Sl13}. It means that the value $\mathcal I_\infty(v)$ is an upper bound of $t^*$ for any $v\in \mathcal C$, $b^\top v=1$.}
\end{remark}

\begin{example}
\emph{We illustrate the previous theoretical results with the following three examples in $\mathbb R$ (i.e., $n=1$), where one can easily verify the validity of $(\mathcal A_1)-(\mathcal A_5)$. In all cases, we set $b=1$. 
More challenging examples are mentioned in Section \ref{subsec_examples_2D}.
\begin{enumerate}
	\item Let $\mathcal I(v)=\frac{1}{2}v^2$, $v\in\mathbb R$. Then $u_t=t$ is the unique minimizer of the function $\mathcal J_t(v)=\mathcal I(v)-tv$ for any $t\in\mathbb R$ and thus $t^*=+\infty$. It is easy to see that $\mathcal I_\infty(v)=+\infty$ for any $v\in\mathbb R$, $v\neq0$. Therefore, $\mathcal C=\{0\}$ and $t_\infty=+\infty$. 
	\item Let $\mathcal I(v)=\frac{1}{2}v^2$ if $|v|\leq1$ and $\mathcal I(v)=|v|-\frac{1}{2}$ if $|v|\geq1$. Then  $u_t=t$ is the unique minimizer of the function $\mathcal J_t(v)=\mathcal I(v)-tv$ for $t\in[0,1)$. If $t=1$ then the solution set $\mathcal K_t=[1,+\infty)$ is unbounded and, thus, $t^*=1$. For $t>1$, no solution exists. Further, we have $\mathcal I_\infty(v)=|v|$ for any $v\in\mathbb R$. Therefore, $\mathcal C=\mathbb R$ and the LA problem has a unique minimum $v=1$, in other words, $t_\infty=\mathcal I_\infty(1)=1$.
	\item Let $\mathcal I(v)=e^{-v}+v-1$, $v\in\mathbb R$. Then, there exists a unique minimum of the function $\mathcal J_t(v)=\mathcal I(v)-tv$ for $t\in[0,1)$. For $t=t^*=1$, no solution exists but the function $\mathcal J_{t^*}$ is bounded from below. Further, $\mathcal I_\infty(v)=+\infty$ if $v<0$, $\mathcal I_\infty(v)=v$ if $v\geq0$ and $\mathcal C=[0,+\infty)$. Hence, $t_\infty=\mathcal I_\infty(1)=1$.
\end{enumerate}}
\end{example}

\begin{example}
\emph{Let $n=1$, $b=1$, $\mathcal I(v)=v-\ln v-2$ if $v\geq 1$ and $\mathcal I(v)=v^2-2v$ if $v\leq1$. Then, it is possible to show that $(\mathcal A_1)-(\mathcal A_4)$ hold but $(\mathcal A_5)$ does not hold. Consequently, $t^*=t_\infty=1$, but $\mathcal J_{t^*}$ is unbounded from below. Therefore, the assumption $(\mathcal A_5)$ is not purposeless.}
\end{example}

\subsection{On solvability of convex optimization problems via LA}
\label{subsec_solvability}

The aim of this subsection is to sketch that the LA approach could lead to innovative solvability conditions for convex optimization problems, not only the ones obeying the assumptions $(\mathcal A_1)-(\mathcal A_5)$.

For the original minimization problem \eqref{J_min} introduced in Section \ref{sec_algebra}, we have the following consequence of Theorems \ref{theorem_K}, \ref{theorem_LA} and \ref{theorem_T*_bound}.

\begin{theorem}
Let the assumptions $(\mathcal A_1)$ -- $(\mathcal A_5)$ be satisfied, $\mathcal K$ be the set of minimizers of $\mathcal J$ defined by \eqref{J_min}, and $t_\infty(=t^*)$ be defined by \eqref{LA_problem}. Then the following statements hold:
\begin{itemize}
    \item $t_\infty>1$ if and only if $\mathcal K$ is bounded and nonempty,
    \item if $t_\infty=1$ then $\mathcal K$ is either empty or unbounded and $\mathcal J$ is bounded from below, 
    \item if $t_\infty<1$ then $\mathcal K$ is empty and $\mathcal J$ is unbounded from below. 
\end{itemize}
\label{theorem_existence}
\end{theorem}

The following examples demonstrate that some of these results could be extended to more general optimization problems than the ones satisfying basic assumption $(\mathcal A_1)$ -- $(\mathcal A_3)$. Such extensions are out of the scope of this paper but they could be an inspiration for researchers.

\begin{example}
    \emph{Consider a linear system of equations $Au=b$, where $A$ is a symmetric, positive semidefinite matrix. It is well-known that this system has a solution if and only if $b\perp\mathrm{Ker}\,A$ (or equivalently, $b\in\mathrm{Im}\,A$). We show that this solvability condition is a consequence of Theorem \ref{theorem_existence} despite the fact that the assumption $(\mathcal A_3)$ does not hold, in general. In particular, we have:
    	$$\mathcal I(v)=\frac{1}{2}v^\top Av,\quad \mathcal I_{\infty}(v)=\lim\limits_{\omega\rightarrow+\infty}\frac{1}{\omega}\mathcal I(\omega v)=\left\{\begin{array}{rl}
		0, &v\in\mathrm{Ker}\,A,\\
		+\infty, & v\not\in\mathrm{Ker}\,A,
	\end{array}\right.,\quad\mathcal C=\mathrm{Ker}\,A,$$
$$t_\infty=\inf_{\substack{v\in\mathrm{Ker}\,A\\ b^\top v=1}}\mathcal I_{\infty}(v)=\left\{
\begin{array}{rl}
0, & b\not\perp\mathrm{Ker}\,A,\\
+\infty, & b\perp\mathrm{Ker}\,A.
\end{array}
\right.$$
Hence, $t_\infty>1$ if and only if $b\perp\mathrm{Ker}\,A$.}
\end{example}

\begin{example}
\emph{Let $n=1$, $b=1$ and $\mathcal I(v)=|v|$ if $v\leq1$, $\mathcal I(v)=2|v|-1$ if $v\geq1$. The function $\mathcal I$ is convex, coercive, satisfies $(\mathcal A_3)-(\mathcal A_5)$, but it is not smooth. Therefore, the original minimization problem \eqref{J_min} cannot be rewritten as the system \eqref{alg_system}. Despite this fact, one can easily verify the following results: $\mathcal K_t=\{0\}$ for any $t\in[0,1)$, $\mathcal K:=\mathcal K_1=[0,1]$, $\mathcal K_t=\{1\}$ for any $t\in(1,2)$, and $\mathcal K_t=[1,\infty)$ for $t=t^*=2$. It means that the problem \eqref{J_min} has the solution set $\mathcal K=[0,1]$, which is bounded and nonempty. Further, $\mathcal I_\infty(v)=-v$ if $v\leq0$, $\mathcal I_\infty(v)=2v$ if $v\geq0$, $\mathcal C=\mathbb R$, and $t_\infty=\mathcal I_{\infty}(1)=2>1$, i.e., the solvability condition from Theorem \ref{theorem_existence} is satisfied. 
Finally, the mapping $\psi:\omega\mapsto t_{\omega}$ defined by Theorem \ref{theorem_omega} is multivalued in this case (e.g., $\psi(1)=[1,2]$). It is caused by the fact that Lemma \ref{lemma_bu} does not hold for nonsmooth functions.}
\end{example}

Finally, it is worth mentioning that the LA analysis problem is meaningful even for constrained optimization problems as it is shown in \cite{SHHC15, CHKS15} for a contact problem of EP bodies.


\section{Parametrization of the function $F$}
\label{sec_SSRM}

In this section, we parametrize the system \eqref{alg_system} such that the function $F$ is replaced by a family of functions $F_\lambda$ depending on a scalar parameter $\lambda$. It is another way how to determine FoS related to the original system \eqref{alg_system}. In Section \ref{subsec_SSRM_def}, we define the parametrized problem and FoS, propose basic assumptions and generalize Theorem \ref{theorem_K}. In Section \ref{subsec_continuation2}, we present a similar continuation method for finding FoS as in Section \ref{sec_LA}. In Section \ref{subsec_LA4SSRM}, we use the LA approach for the determination of FoS. In Section \ref{subsec_examples_2D}, we illustrate theoretical results with examples having solutions in closed forms. 

\subsection{Definition, assumptions and analysis}
\label{subsec_SSRM_def}

Let $\lambda_0\in[0,1]$ be given. For any scalar factor $\lambda\geq\lambda_0$, consider a vector-valued function $F_\lambda\colon\mathbb R^n\rightarrow\mathbb R$ and assume that the original function $F$ corresponds to $\lambda=1$. The parametrized problem and FoS are defined as follows:
\begin{equation}
	\mbox{given $\lambda\geq \lambda_0,\;$ find } u_\lambda\in\mathbb R^n: \quad F_\lambda(u_\lambda)=b,
	\label{system_lambda}
\end{equation}
\begin{equation}
	\lambda^*\;:=\;\mbox{{supremum} of } \; \bar \lambda\geq \lambda_0\;\;\mbox{such that $\mathcal K_\lambda$ is nonempty for any $\lambda\in[\lambda_0,\bar \lambda]$,}
	\label{FoS_SSRM}
\end{equation}
where $\mathcal K_\lambda$ denotes the solution set to \eqref{system_lambda}. The definition admits the case $\lambda^*=+\infty$. Let us note that the parametrization of $b$ from Section \ref{sec_LA} is a special case of the $\lambda$-parametrization because it corresponds to the choice $F_\lambda:=F/\lambda$. Therefore, the $\lambda$-parametrization is more involved and thus more complex assumptions have to be proposed. All the assumptions introduced below can be verified for $F_\lambda:=F/\lambda$ if $(\mathcal A_1)-(\mathcal A_5)$ hold.


We start with the following basic assumptions:
\begin{itemize}
	\item[$(\mathcal B_1)_\lambda$] $b\neq0$, $F_\lambda(0)=0$ and $F_\lambda$ is continuous in $\mathbb R^n$ for any $\lambda\geq \lambda_0$.
	\item[$(\mathcal B_2)_\lambda$] $F_\lambda$ has a convex potential for any $\lambda\geq \lambda_0$, i.e., there exists a convex and continuously differentiable function $\mathcal I_\lambda\colon\mathbb R^n\rightarrow \mathbb R$ such that $\nabla\mathcal I_\lambda(v)=F_\lambda(v)$ and $\mathcal I_\lambda(0)=0$.
    \item[$(\mathcal B_3)_\lambda$] (Coercivity of $\{\mathcal J_\lambda\}_{\lambda\geq\lambda_0}$.)
\begin{itemize}
    \item[(i)] The function $\mathcal J_{\lambda_0}$ is coercive in $\mathbb R^n$.
    \item[(ii)] If $\mathcal J_{\bar\lambda}$ is bounded from below in $\mathbb R^n$ for some $\bar\lambda>\lambda_0$, then $\mathcal J_{\lambda}$ is coercive in $\mathbb R^n$ for any $\lambda\in[\lambda_0,\bar\lambda)$.
    \item[(iii)] If $\mathcal J_{\bar\lambda}$ is coercive in $\mathbb R^n$ for some $\bar\lambda>\lambda_0$ then there exists $\epsilon>0$ such that $\mathcal J_{\bar\lambda+\epsilon}$ is also coercive.
\end{itemize}
\item[$(\mathcal B_3)^+_\lambda$] The assumption $(\mathcal B_3)_\lambda$ holds and, in addition, if $\mathcal J_{\bar\lambda}$ is unbounded from below in $\mathbb R^n$ for some $\bar\lambda>\lambda_0$ then there exists $\epsilon>0$ such that $\mathcal J_{\bar\lambda-\epsilon}$ is also unbounded from below.	
\end{itemize}
We see that $(\mathcal B_1)_\lambda$ and $(\mathcal B_2)_\lambda$ straightforwardly extend the assumptions $(\mathcal A_1)$ and $(\mathcal A_2)$ from Section \ref{sec_algebra}. Due to $(\mathcal B_2)_\lambda$, one can introduce the following minimization problem:
\begin{equation}
	\mathcal J_\lambda(u_\lambda)\leq \mathcal J_\lambda(v)\quad \text{for all } v\in \mathbb R^n,\quad \mathcal J_\lambda(v):=\mathcal I_\lambda(v)-b^\top v.
	\label{system_lambda2}
\end{equation}
The solution set to \eqref{system_lambda2} coincides with $\mathcal K_\lambda$ corresponding to \eqref{system_lambda}. 

The assumption $(\mathcal B_3)_\lambda$ is much more general than  $(\mathcal A_3)$ because it is related to the whole cost function $\mathcal J_\lambda$, not only to its part $\mathcal I_\lambda$. Therefore, $(\mathcal B_3)_\lambda$ depends on $b$, in general. Although it is possible to replace $(\mathcal B_3)_\lambda$ with more specific assumptions, which are independent of $b$, such a treatment would be insufficient for geotechnical practice as is mentioned in Remark \ref{remark_SSR} and illustrated with Example \ref{examp2}, see below. 

The following result extends Theorems \ref{theorem_K} and \ref{theorem_T*_bound}. 

\begin{theorem}
	Let $(\mathcal B_1)_\lambda$, $(\mathcal B_2)_\lambda$ and $(\mathcal B_3)_\lambda$ be satisfied. Then $\lambda^*>\lambda_0$, $\mathcal K_\lambda$ is nonempty and bounded for any $\lambda\in[\lambda_0,\lambda^*)$, $\mathcal K_{\lambda^*}$ is either unbounded or empty for $\lambda^*<+\infty$ and $\mathcal J_{\lambda}$ is unbounded from below for $\lambda>\lambda^*$. Additionally, if $(\mathcal B_3)^+_\lambda$ also holds and $\lambda^*<+\infty$ then $\mathcal J_{\lambda^*}$ is bounded from below.
	\label{theorem_K_lambda}
\end{theorem}

\begin{proof}
This result is in fact a consequence of Lemma \ref{lemma_K_t}. In particular, $\lambda^*>\lambda_0$ follows from $(\mathcal B_3)_\lambda$ (i) and (iii). Let $\bar\lambda<\lambda^*$. Then $\mathcal K_{\bar\lambda}\neq\emptyset$ (as follows from the definition \eqref{FoS_SSRM} of $\lambda^*$) and thus $\mathcal J_{\bar\lambda}$ is bounded from below in $\mathbb R^n$. By, $(\mathcal B_3)_\lambda$ (ii), $\mathcal J_{\lambda}$ is coercive in $\mathbb R^n$ for any $\lambda\in[\lambda_0,\bar\lambda)$. Consequently, $\mathcal K_\lambda$ is nonempty and bounded for any $\lambda\in[\lambda_0,\lambda^*)$. If $\lambda^*<+\infty$ then $\mathcal K_{\lambda^*}$ cannot be nonempty and bounded as follows from $(\mathcal B_3)_\lambda$ (iii) and \eqref{FoS_SSRM}. In addition, if $(\mathcal B_3)^+_\lambda$ holds then $\mathcal J_{\lambda^*}$ cannot be unbounded from below because we would obtain a contradiction to \eqref{FoS_SSRM}.
\end{proof}

\begin{remark}
(EP interpretation.) \emph{The $\lambda$-parametrization is inspired by the SSR method mentioned in Section \ref{sec_intro}. In particular, the factor $\lambda$ is used to reduce selected material parameters. The key assumption $(\mathcal B_2)_\lambda$ corresponds to the optimization variant of the SSR methods suggested in \cite{Sysala_2021}.  $(\mathcal B_3)_\lambda$ and $(\mathcal B_3)^+_\lambda$ are expected within the SSR method on the basis of numerical observations, but they cannot be verified a priori, in general. Further, the vector $b$ usually represents in geotechnics external forces like the gravity, which cause compressive stresses in rock or soil massifs. The SSR method was proposed just for such external forces and need not be meaningful for arbitrary $b$ unlike the LL and LA methods. This fact motivated us to formulate $(\mathcal B_3)_\lambda$ and $(\mathcal B_3)^+_\lambda$ such that the influence of $b$ is taken into account. 
}
\label{remark_SSR}
\end{remark}

\subsection{Indirect continuation technique}
\label{subsec_continuation2}

We introduce and analyze a similar continuation method as in Section \ref{sec_LA}. Besides $(\mathcal B_1)_\lambda-(\mathcal B_3)_\lambda$, the following additional assumptions are needed:
\begin{itemize}
\item[$(\mathcal B_4)_\lambda$] For any $\lambda\geq\lambda_0$, there exist constants $c_{1,\lambda}>0$ and $c_{2,\lambda}\geq0$ such that the following estimate holds:
	\begin{equation}
		\mathcal I_{\lambda}(v)\geq c_{1,\lambda}\|v\|-c_{2,\lambda}\quad\text{for all } v\in\mathbb R^n.
		\label{lin_growth_lambda}
	\end{equation}
	\item[$(\mathcal B_5)_\lambda$] For any $v\in\mathbb R^n$, the function $\lambda\mapsto\mathcal I_{\lambda}(v)$ is nonincreasing, i.e.,
	\begin{equation}
\mathcal I_{\tilde\lambda}(v)\geq \mathcal I_\lambda(v)\quad\mbox{for all } \lambda,\tilde\lambda\geq\lambda_0,\; \tilde\lambda\leq\lambda.
\label{I_lambda}
	\end{equation}
\item[$(\mathcal B_6)_\lambda$] For any sequences $\{\lambda_n\}\geq\lambda_0$ and $\{v_n\}\subset\mathbb R^n$ such that $\lambda_n\rightarrow\lambda$ and $v_n\rightarrow v$ as $n\rightarrow+\infty$, we have
	\begin{equation}
		\lim_{n\rightarrow+\infty}\mathcal I_{\lambda_n}(v_n)=\mathcal I_{\lambda}(v).
		\label{I_lambda_n}
	\end{equation}
	\item[$(\mathcal B_7)_\lambda$] For any $\lambda\geq\lambda_0$ there exists $v_\lambda\in\mathbb R^n$ such that
	\begin{equation}
		\lim_{\lambda\rightarrow+\infty}\mathcal J_{\lambda}(v_\lambda)=-\infty.
		\label{v_lambda}
	\end{equation}
	\item[$(\mathcal B_8)_\lambda$] Let $\lambda_1<\lambda_2$ and $\mathcal K_{\lambda_i}\neq\emptyset$, $i=1,2$. Then  $b^\top u_1\leq b^\top u_2$ for any $u_1\in\mathcal K_{\lambda_1}$ and $u_2\in\mathcal K_{\lambda_2}$. 
	\item[$(\mathcal B_8)^+_\lambda$] Let $\lambda_1<\lambda_2$ and $\mathcal K_{\lambda_i}\neq\emptyset$, $i=1,2$. Then  $b^\top u_1< b^\top u_2$ for any $u_1\in\mathcal K_{\lambda_1}$ and $u_2\in\mathcal K_{\lambda_2}$. 
\end{itemize} 

\begin{remark} (EP interpretation.)
\emph{The assumption $(\mathcal B_4)_\lambda-(\mathcal B_6)_\lambda$ are natural and can be verified a priori within the SSR method.  $(\mathcal B_7)_\lambda$ is trivially satisfied if $\lambda^*<+\infty$ and the case $\lambda^*=+\infty$ is not usual in geotechnical practice. Further, $(\mathcal B_8)_\lambda$ and $(\mathcal B_8)^+_\lambda$ are expected if $\lambda_0\in(0,1]$ is chosen sufficiently large as in Section \ref{sec_num_example}. However, the problem $(\mathcal P)_\lambda$ may have a purely elastic solution for too small values of $\lambda$ and this solution is independent of $\lambda$. Therefore, $(\mathcal B_8)^+_\lambda$ does not hold for too small $\lambda_0$.}
\end{remark}

The following lemmas enable us to analyze the mapping $\lambda\mapsto b^\top u_\lambda$, where $u_\lambda\in\mathcal K_\lambda$. 
\begin{lemma}
Let $(\mathcal B_1)_\lambda-(\mathcal B_7)_\lambda$ be satisfied. Then 
\begin{equation}
	\sup\Big\{b^\top u_\lambda\ |\;\;u_\lambda\in\bigcup_{\lambda\geq\lambda_0}\mathcal K_\lambda\Big\}=+\infty.
	\label{bu_lim}
\end{equation}
\label{lemma_omega_infty}
\end{lemma}

\begin{proof}
Let us remind that $\mathcal K_\lambda=\emptyset$ for any $\lambda\in[\lambda_0,\lambda^*)$. In account of this fact and Theorem \ref{theorem_K_lambda}, it suffices to distinguish the following three possible scenarios.

$(a)$ \textit{Let $\lambda^*<+\infty$ and $\mathcal K_{\lambda^*}$ be unbounded.} Since $\mathcal K_{\lambda^*}$ is also closed and convex, there exist $u^*\in\mathcal K_{\lambda^*}$ and $w^*\in\mathbb R^n$, $w^*\neq0$, such that $u^*+\beta w^*\in\mathcal K_{\lambda^*}$ for any $\beta\geq0$. From \eqref{system_lambda}, we have
$$\mathcal I_{\lambda^*}(u^*)-b^\top u^*=\mathcal I_{\lambda^*}(u^*+\beta w^*)-b^\top (u^*+\beta w^*)\quad\forall \beta\geq0.$$
Hence,
\begin{align*}
b^\top w^*&=\lim_{\beta\rightarrow+\infty}\frac{1}{\beta}[\mathcal I_{\lambda^*}(u^*+\beta w^*)-\mathcal I_{\lambda^*}(u^*)]\\
&\stackrel{(\mathcal B_4)_\lambda}{\geq}\lim_{\beta\rightarrow+\infty}\frac{1}{\beta}[c_{1,\lambda^*}\|u^*+\beta w^*\|-c_{2,\lambda^*}-\mathcal I_{\lambda^*}(u^*)]\geq c_{1,\lambda^*}\|w^*\|>0,
\end{align*}
so that $\lim_{\beta\rightarrow+\infty}b^\top (u^*+\beta w^*)=+\infty$ and thus \eqref{bu_lim} holds in this case.

$(b)$ \textit{Let $\lambda^*<+\infty$ and $\mathcal K_{\lambda^*}=\emptyset$.} Suppose that there exist sequences $\{\lambda_n\}$ and $\{u_n\}$ such that $\lambda_n<\lambda^*$, $\lambda_n\rightarrow\lambda^*$ as $n\rightarrow+\infty$, $u_n\in\mathcal K_{\lambda_n}$ and $\{b^\top u_n\}$ is bounded from above. Then
$$0\geq \mathcal J_{\lambda_n}(u_n)=\mathcal I_{\lambda_n}(u_n)-b^\top u_n\stackrel{(\mathcal B_5)_\lambda}{\geq}\mathcal I_{\lambda^*}(u_n)-b^\top u_n\stackrel{(\mathcal B_4)_\lambda}{\geq}c_{1,\lambda^*}\|u_n\|-c_{2,\lambda^*}-c,$$
where $c>0$ is an upper bound of $\{b^\top u_n\}$. Hence, the sequence $\{u_n\}$ is bounded and its accumulation points belong to $\mathcal K_{\lambda^*}$ which can be proven using $(\mathcal B_6)_\lambda$. This is a contradiction to the assumption $\mathcal K_{\lambda^*}=\emptyset$. Therefore, \eqref{bu_lim} holds again.

$(c)$ \textit{Let $\lambda^*=+\infty$.} Then for any sequence $\{u_\lambda\}$ of solutions from to $\mathcal K_\lambda$ and any sequence $\{v_\lambda\}$ satisfying $(\mathcal B_7)_\lambda$, we have:
$$-\lim_{\lambda\rightarrow+\infty}b^\top u_\lambda\leq\lim_{\lambda\rightarrow+\infty}\mathcal J_\lambda(u_\lambda)\leq \lim_{\lambda\rightarrow+\infty}\mathcal J_\lambda(v_\lambda)=-\infty.$$
Hence, $b^\top u_\lambda\rightarrow+\infty$ as $\lambda\rightarrow+\infty$, which completes the proof.
\end{proof}	

\begin{lemma}
Let $(\mathcal B_1)_\lambda-(\mathcal B_8)_\lambda$ be satisfied. Let $\lambda\in(\lambda_0,\lambda^*]\cap\mathbb R$ be such that $\mathcal K_\lambda\neq\emptyset$, $\{\lambda_n\}\subset[\lambda_0,\lambda)$ be an increasing sequence tending to $\lambda$ and $\{u_n\}$ be a sequence of solutions belonging to $\mathcal K_{\lambda_n}$. Then 
\begin{equation}
\lim_{n\rightarrow+\infty} b^\top u_n=\min_{u_\lambda\in\mathcal K_\lambda}b^\top u_\lambda.
\label{bu_lambda}
\end{equation}
\label{lemma_semicontinuity1}
\end{lemma}

\begin{proof}
We have the following estimates:
\begin{align*}
0&\geq\mathcal J_{\lambda_n}(u_n)\stackrel{(\mathcal B_5)_\lambda}{\geq}\mathcal J_{\lambda}(u_n)=\mathcal I_{\lambda}(u_n)-b^\top u_n\stackrel{(\mathcal B_4)_\lambda,(\mathcal B_8)_\lambda}{\geq} c_{1,\lambda}\|u_n\|-c_{2,\lambda}-b^\top u_\lambda,
\end{align*}
where $u_\lambda\in\mathcal K_\lambda$. Hence, the sequence $\{u_n\}$ is bounded and its accumulation points belong to $\mathcal K_{\lambda}$ which can be proven using $(\mathcal B_6)_\lambda$. Let $\bar u_\lambda\in\mathcal K_\lambda$ denote an arbitrary accumulation point of $\{u_n\}$. Then, from the assumption $(\mathcal B_8)_\lambda$, it is easy to see that $b^\top \bar u_\lambda=\min_{u_\lambda\in\mathcal K_\lambda}b^\top u_\lambda$, which gives \eqref{bu_lambda}.
\end{proof}

\begin{lemma}
	Let $(\mathcal B_1)_\lambda-(\mathcal B_8)_\lambda$ be satisfied. Let $\lambda\in[\lambda_0,\lambda^*)$, $\{\lambda_n\}\subset(\lambda,\lambda^*)$ be a decreasing sequence tending to $\lambda$ and $\{u_n\}$ be a sequence of solutions belonging to $\mathcal K_{\lambda_n}$. Then 
	\begin{equation}
		\lim_{n\rightarrow+\infty} b^\top u_n=\max_{u_\lambda\in\mathcal K_\lambda}b^\top u_\lambda.
	\end{equation}
\label{lemma_semicontinuity2}
\end{lemma}
\begin{proof}
Since $\lambda_n<\ldots<\lambda_2<\lambda_1<\lambda^*$, we have:
$$0\geq\mathcal J_{\lambda_n}(u_n)\stackrel{(\mathcal B_5)_\lambda}{\geq}\mathcal J_{\lambda_1}(u_n)=\mathcal I_{\lambda_1}(u_n)-b^\top u_n\stackrel{(\mathcal B_4)_\lambda,(\mathcal B_8)_\lambda}{\geq} c_{1,\lambda_1}\|u_n\|-c_{2,\lambda_1}-b^\top u_1.$$
Hence, the sequence $\{u_n\}$ is bounded. The rest of the proof is analogous to the previous one.
\end{proof}

Now, we formulate and prove the main result, which is analogous to Theorem \ref{theorem_omega}. We need to define the following value:
\begin{equation}
	\omega_0:=\min_{u\in\mathcal K_{\lambda_0}}b^\top u.
	\label{omega_0}
\end{equation}

\begin{theorem}
Let $(\mathcal B_1)_\lambda-(\mathcal B_8)_\lambda$ be satisfied. Then for any $\omega\geq\omega_0$ there exist  $\lambda_\omega\in[\lambda_0,\lambda^*]\cap\mathbb R$ and $u_\omega\in\mathbb R^n$ such that 
\begin{equation}
F_{\lambda_\omega}(u_\omega)=b,\quad b^\top u_\omega=\omega.
\label{system_omega}
\end{equation}
The mapping $\psi\colon\omega\mapsto\lambda_\omega$ (generally multi-valued) is nondecreasing and $\psi(\omega)\rightarrow\lambda^*$ as $\omega\rightarrow+\infty$. If $(\mathcal B_8)^+_\lambda$ is satisfied then $\psi$ is a single-valued and continuous function.
\label{theorem_omega_SSRM}
\end{theorem}

\begin{proof}
First, we prove the solvability of \eqref{system_omega} for arbitrary $\omega\geq\omega_0$. Define
\begin{align*}
\lambda_1&:=\sup\big\{\lambda\geq\lambda_0\ |\; \omega\geq\min_{u_{\lambda}\in\mathcal K_{\lambda}}[b^\top u_{\lambda}]\big\},\\
\lambda_2&:=\max\Big\{\lambda_1,\;\inf\big\{\lambda\geq\lambda_0\ |\; \omega\leq\sup_{u_{\lambda}\in\mathcal K_{\lambda}}[b^\top u_{\lambda}]\big\}\Big\}.
\end{align*}
By Lemma \ref{lemma_omega_infty}, there exists $\lambda\in[\lambda_0,\lambda^*]\cap\mathbb R$ and $u_{\lambda}\in\mathcal K_{\lambda}$ such that $b^\top u_{\lambda}\geq \omega$. Therefore, the values $\lambda_1$ and $\lambda_2$ are finite and $\lambda_1\leq\lambda_2\leq\lambda^*$. Suppose that $\lambda_1<\lambda_2$. Then there exists $\lambda\in(\lambda_1,\lambda_2)$. Since $\lambda<\lambda^*$, Theorem \ref{theorem_K_lambda} implies $\mathcal K_\lambda\neq\emptyset$. By the definitions of $\lambda_1$ and $\lambda_2$, we have
$$\sup_{u_{\lambda}\in\mathcal K_{\lambda}}[b^\top u_{\lambda}]<\omega<\min_{u_{\lambda}\in\mathcal K_{\lambda}}[b^\top u_{\lambda}],$$
which is a contradiction. Therefore, $\lambda_1=\lambda_2=:\lambda_\omega$. From the definition of $\lambda_\omega$ and Lemma \ref{lemma_semicontinuity1}, it follows that the solution set $\mathcal K_{\lambda_\omega}$ is nonempty and there exist $u_1,u_2\in\mathcal K_{\lambda_\omega}$ such that $\omega_1\leq\omega\leq\omega_2$, where $\omega_i:=b^\top u_i$, $i=1,2$. Let
$$u_\omega=\frac{\omega_2-\omega}{\omega_2-\omega_1}u_1+\frac{\omega-\omega_1}{\omega_2-\omega_1}u_2.$$
Then, it is easy to verify that $b^\top u_\omega=\omega$. In addition, $u_\omega$ is a convex combination of $u_1$ and $u_2$. Therefore, $u_\omega\in\mathcal K_{\lambda_\omega}$ and thus the pair $(u_\omega,\lambda_\omega)$ satisfies \eqref{system_omega}.

The mapping $\psi\colon\omega\mapsto\lambda_\omega$ is nondecreasing as follows from $(\mathcal B_8)_\lambda$. From Lemma \ref{lemma_omega_infty} and its proof, we have $\lim_{\omega\rightarrow+\infty}\psi(\omega)=\lambda^*$. Under the assumption $(\mathcal B_8)^+_\lambda$, the function $\psi$ is single-valued. Continuity of $\psi$ follows from Lemmas \ref{lemma_semicontinuity1} and \ref{lemma_semicontinuity2}.
\end{proof}

Theorem \ref{theorem_omega_SSRM} enables us to determine the safety factor $\lambda^*$ indirectly, using the parameter $\omega$ and its increase to $+\infty$. The main advantage of this indirect continuation is the fact that the extended system \eqref{system_omega} has a solution for any $\omega\geq\omega_0$ unlike the original system \eqref{system_lambda}. 

\subsection{Limit analysis approach for the safety factor $\lambda^*$}
\label{subsec_LA4SSRM}

To employ the LA approach from Section \ref{subsec_LA} for the determination of $\lambda^*$ defined by \eqref{FoS_SSRM}, we need to combine the parametrizations by $\lambda$ and $t$. Consider the following system of nonlinear equations for given $\lambda\geq\lambda_0$ and $t\geq0$:
\begin{equation}
	\mbox{find } u_{\lambda,t}\in\mathbb R^n: \quad F_\lambda(u_{\lambda,t})=tb,
	\label{system_lambda_t}
\end{equation}
and the corresponding minimization problem: 
\begin{equation}
	\mathcal J_{\lambda,t}(u_{\lambda,t})\leq \mathcal J_{\lambda,t}(v)\quad \text{for all } v\in \mathbb R^n,\quad \mathcal J_{\lambda,t}(v):=\mathcal I_\lambda(v)-tb^\top v.
	\label{system_lambda2_t}
\end{equation}
For any $\lambda\geq\lambda_0$, we define the following FoS, which is analogous to the one from Section \ref{sec_LA}:
\begin{align}
	\ell(\lambda)&:=\;\mbox{{supremum} of } \; \bar t\geq 0\;\;\mbox{such that $\mathcal K_{\lambda,t}$ is nonempty for any $t\in[0,\bar t]$,}
	\label{ell_lambda}
\end{align}
where $\mathcal K_{\lambda,t}$ denotes the solution set to \eqref{system_lambda_t} (or to \eqref{system_lambda2_t}). The definition again admits the case $\ell(\lambda)=+\infty$ for some $\lambda\geq\lambda_0$. It is easy to see that the assumptions $(\mathcal B_1)_\lambda$, $(\mathcal B_2)_\lambda$ and $(\mathcal B_4)_\lambda$ are sufficient to extend the results from Section \ref{sec_LA} on the safety factor $\ell(\lambda)$ for any $\lambda\geq\lambda_0$. From $(\mathcal B_5)_\lambda$ and $(\mathcal B_3)_\lambda$, it follows that the function $\ell$ is nonincreasing and $\ell(\lambda_0)>1$.

Next, we need to modify $(\mathcal B_3)^+_\lambda$ to be valid for a family of the functions $\mathcal J_{\lambda,t}$:
\begin{itemize}
\item[$(\mathcal B_3)_{\lambda,t}^+$] Let $t\in[0,\ell(\lambda_0)]\cap\mathbb R$. Then:
\begin{itemize}
    \item[(i)] The function $\mathcal J_{\lambda_0,1}=\mathcal J_{\lambda_0}$ is coercive in $\mathbb R^n$.
    \item[(ii)] If $\mathcal J_{\bar\lambda,t}$ is bounded from below in $\mathbb R^n$ for some $\bar\lambda>\lambda_0$ then $\mathcal J_{\lambda,t}$ is coercive in $\mathbb R^n$ for any $\lambda\in[\lambda_0,\bar\lambda)$.
    \item[(iii)] If $\mathcal J_{\bar\lambda,t}$ is coercive in $\mathbb R^n$ for some $\bar\lambda>\lambda_0$ then there exists $\epsilon>0$ such that $\mathcal J_{\bar\lambda+\epsilon,t}$ is also coercive.
    \item[(iv)] If $\mathcal J_{\bar\lambda,t}$ is unbounded from below in $\mathbb R^n$ for some $\bar\lambda>\lambda_0$ then there exists $\epsilon>0$ such that $\mathcal J_{\bar\lambda-\epsilon,t}$ is also unbounded.
\end{itemize}
\end{itemize} 

Finally, in order to use the LA approach from Section \ref{subsec_LA}, we introduce the following notation for any $\lambda\geq\lambda_0$:
\begin{equation}
	\mathcal I_{\omega,\lambda}\colon \mathbb R^n\rightarrow \mathbb R,\quad \mathcal I_{\omega,\lambda}(v):=\frac{1}{\omega}\mathcal I_{\lambda}(\omega v),\quad v\in\mathbb R^n,\;\omega>0,
	\label{I_omega_lambda}
\end{equation}
\begin{equation}
	\mathcal I_{\infty,\lambda}\colon \mathbb R^n\rightarrow \mathbb R_+\cup\{+\infty\},\quad \mathcal I_{\infty,\lambda}(v):=\lim_{\omega\rightarrow+\infty}\mathcal I_{\omega,\lambda}(v),\quad v\in\mathbb R^n,
	\label{I_infty_lambda}
\end{equation}
\begin{equation}
	\mathcal C_\lambda:=\mathrm{dom}\, \mathcal I_{\infty,\lambda}=\{v\in\mathbb R^n\ |\;\; \mathcal I_{\infty,\lambda}(v)<+\infty\},
	\label{C_lambda}
\end{equation}
\begin{equation}
	\ell_\infty(\lambda):=\inf_{\substack{v\in\mathbb R^n\\ b^\top v=1}}\ \mathcal I_{\infty,\lambda}(v)=\inf_{\substack{v\in\mathcal C_\lambda\\ b^\top v=1}}\ \mathcal I_{\infty,\lambda}(v),
	\label{LA_problem_lambda}
\end{equation}
and modify the assumptions $(\mathcal A_4)$ and $(\mathcal A_5)$ from Section \ref{subsec_LA} to be valid for any $\lambda\geq\lambda_0$:
\begin{itemize}
	\item[$(\mathcal B_9)_\lambda$] \textit{For any sequence $\{v_\omega\}\subset\mathbb R^n$ such that $\{\mathcal I_{\omega,\lambda}(v_\omega)\}$ is bounded and $v_\omega\rightarrow v_\infty$ as $\omega\rightarrow+\infty$, the following properties hold:} 
	\begin{equation}
		v_\infty \in\mathcal C_\lambda\quad\mbox{and}\quad \liminf_{\omega\rightarrow+\infty}\mathcal I_{\omega,\lambda}(v_\omega)\geq\mathcal I_{\infty,\lambda}(v_\infty).
		\label{A4_lambda}
	\end{equation}
	\item[$(\mathcal B_{10})_\lambda$] \textit{For any $v\in\mathcal C_\lambda$, there exists $c_{v,\lambda}>0$ such that}
	\begin{equation}
		\omega\mathcal I_{\infty,\lambda}(v)-c_{v,\lambda}\leq \mathcal I_\lambda(\omega v)\leq \omega\mathcal I_{\infty,\lambda}(v)\quad\mbox{for all } \omega\geq0.
		\label{I_infty_bound_lambda}
	\end{equation}
\end{itemize}
From $(\mathcal B_{1})_\lambda$, $(\mathcal B_{2})_\lambda$, $(\mathcal B_{4})_\lambda$, $(\mathcal B_{9})_\lambda$ and $(\mathcal B_{10})_\lambda$, it follows that $\ell(\lambda)=\ell_\infty(\lambda)$ for any $\lambda\geq\lambda_0$. Further, we have the following result.

\begin{theorem}
Let the assumptions $(\mathcal B_1)_\lambda$, $(\mathcal B_{2})_\lambda$, $(\mathcal B_3)_{\lambda,t}^+$, $(\mathcal B_{4})_\lambda$, $(\mathcal B_{5})_\lambda$, $(\mathcal B_{9})_\lambda$ and $(\mathcal B_{10})_\lambda$ be satisfied. Then the function $\ell$ defined by \eqref{ell_lambda} is continuous and decreasing for any $\lambda\geq\lambda_0$ such that $\ell(\lambda)<+\infty$ and $\ell(\lambda)>1$ for any $\lambda\in[\lambda_0,\lambda^*)$. In addition, if $\lambda^*<+\infty$, then $\lambda^*$ is a unique solution of the equation $\ell(\lambda)=1$.
\end{theorem}

\begin{proof}
Let $\lambda_1,\lambda_2\geq\lambda_0$ be such that $\lambda_1<\lambda_2$ and $\ell(\lambda_i)<+\infty$, $i=1,2$. From Theorem \ref{theorem_T*_bound} (applied to $\mathcal I_{\lambda_2}$ instead of $\mathcal I$), the function $\mathcal J_{\lambda_2,\ell(\lambda_2)}$ is bounded from below. Therefore, $\mathcal J_{\lambda_1,\ell(\lambda_2)}$ is coercive as follows from $(\mathcal B_3)_{\lambda,t}^+$. By Theorem \ref{theorem_K} and its proof (with $\mathcal I_{\lambda_1}$ instead of $\mathcal I$), we obtain $\ell(\lambda_1)>\ell(\lambda_2)$. Therefore, $\ell$ is decreasing in its effective domain. 

In view of Theorem \ref{theorem_K_lambda}, the function $\mathcal J_{\lambda,1}=\mathcal J_{\lambda}$ is coercive for any $\lambda\in[\lambda_0,\lambda^*)$. Then, from Theorem \ref{theorem_K} and its proof (with $\mathcal I_{\lambda}$ instead of $\mathcal I$), it follows that there exists $\epsilon>0$ such that $\mathcal J_{\lambda,1+\epsilon}$ is also coercive. Thus $\ell(\lambda)>1$ for any $\lambda\in[\lambda_0,\lambda^*)$. If $\lambda^*<+\infty$ then $\mathcal J_{\lambda,1}$ is unbounded for any $\lambda>\lambda^*$ as follows from the definition of $\lambda^*$. From Theorem \ref{theorem_T*_bound} (with $\mathcal I_{\lambda}$ instead of $\mathcal I$), the function $\mathcal J_{\lambda,1-\epsilon}$ is also unbounded for sufficiently small $\epsilon>0$ and thus $\ell(\lambda)<1$ for any $\lambda>\lambda^*$.

The proof of continuity of $\ell$ is split into two parts. First, let $\{\lambda_n\}$ be a decreasing sequence converging to $\lambda\geq\lambda_0$, where $\ell(\lambda)<+\infty$. Since $\ell$ is a decreasing function, we have $\ell(\lambda)>\ell(\lambda_n)$ for any $n>0$. From the definition of $\ell(\lambda)$, it follows that the function $\mathcal J_{\lambda,t}$ is coercive for any $t<\ell(\lambda)$. Due to $(\mathcal B_3)_{\lambda,t}^+$, there exists $m\in\mathbb N$ such that $\mathcal J_{\lambda_m,t}$ is also coercive. Hence, 
$$\lim_{m\rightarrow+\infty}\ell(\lambda_m)\geq t\quad\mbox{for all }t<\ell(\lambda),$$
i.e., 
$$\lim_{m\rightarrow+\infty}\ell(\lambda_m)\geq \ell(\lambda)>\ell(\lambda_n)\quad\mbox{for all }n\in\mathbb N.$$
Therefore, $\ell(\lambda_n)\rightarrow\ell(\lambda)$ as $n\rightarrow+\infty$.

Second, let $\{\lambda_n\}$ be an increasing sequence such that  $\ell(\lambda_n)<+\infty$ and $\lambda_n\rightarrow\lambda$ as $n\rightarrow+\infty$. Since $\ell$ is a decreasing function, we have $\ell(\lambda_n)>\ell(\lambda)$. Denote 
\begin{equation}
\bar t:=\lim_{n\rightarrow+\infty}\ell(\lambda_n)\geq\ell(\lambda).
\label{bar_t}
\end{equation}
Since $\ell(\lambda_n)>\bar t$, the function $\mathcal J_{\lambda_n,\bar t}$ is coercive. From $\lambda_n\rightarrow\lambda$, $(\mathcal B_3)_{\lambda,t}^+$ and Theorem \ref{theorem_K_lambda} (applied to $\mathcal J_{\lambda,\bar t}$ instead of $\mathcal J_{\lambda}$), we obtain that $\mathcal J_{\lambda,\bar t}$ is bounded from below. Then, Theorem \ref{theorem_K} (with $\mathcal I_\lambda$ instead of $\mathcal I$) yields $\ell(\lambda)\geq\bar t$. Therefore, the equality holds in \eqref{bar_t}.

From the derived properties of $\ell$, it is easy to see that if $\lambda^*<+\infty$, then $\lambda^*$ is a unique solution of the equation $\ell(\lambda)=1$.
\end{proof}

The function $\ell$ and its properties are sketched in Figure \ref{fig_ell}. 

\begin{figure}[h]
	\centering
	\begin{picture}(150,140)
		
		\put(10,10){\vector(1,0){140}}
		\put(10,10){\vector(0,1){130}}
		
		\put(8,70){\line(1,0){4}}
		\put(104,8){\line(0,1){4}}
		\put(70,93){\circle*{3}}
		\put(8,93){\line(1,0){4}}
		\put(70,8){\line(0,1){4}}
		\put(150,6){\makebox(0,0)[t]{$\lambda$}}
		
		\put(70,6){\makebox(0,0)[t]{$\lambda_0$}}
		\put(7,93){\makebox(0,0)[r]{$\ell(\lambda_0)$}}
		\put(142,40){\makebox(0,0)[l]{$\ell$}}
	\qbezier(20,120)(100,80)(140,40)
	\put(142,40){\makebox(0,0)[l]{$\ell$}}

			\put(104,70){\circle*{3}}
			\put(8,70){\line(1,0){4}}
			\put(104,8){\line(0,1){4}}
			\put(104,6){\makebox(0,0)[t]{$\lambda^*$}}
			\put(7,70){\makebox(0,0)[r]{1}}
			\multiput(10,70)(4,0){32}{\line(1,0){2}} 
			\multiput(104,10)(0,4){15}{\line(0,1){2}}
					
	\end{picture}
	\caption{Visualization of the function $\ell$, which relates $\lambda^*$ with LA.}
	\label{fig_ell}
\end{figure}

\subsection{Examples with solutions in closed forms}
\label{subsec_examples_2D}

The previous assumptions and results are now illustrated with two examples, which mimic different EP models. The corresponding functions $\mathcal I_{\lambda}$ have the following uniform definition:
\begin{equation}
\mathcal I_{\lambda}(v)=\max_{x\in M_\lambda}\Big\{x^\top v-\frac{1}{2}\|x\|^2\Big\},\quad \lambda\geq\lambda_0,
\label{example2}
\end{equation}
where $M_\lambda\subset\mathbb R^n$ is a closed, convex, nonempty set and $0\in\mathrm{int}\, M_\lambda$ for any $\lambda\geq\lambda_0$. It is well-known (see e.g. \cite{Mo65}) that the function $\mathcal I_{\lambda}$ is convex, coercive, continuously differentiable, $\mathcal I_\lambda(0)=0$ and $F_\lambda=\nabla \mathcal I_{\lambda}$ is a projection of $\mathbb R^n$ onto $M_\lambda$. Therefore, $F_\lambda(b)=b$ (i.e., $b\in\mathcal K_\lambda$) holds if and only if $b\in M_\lambda$. In addition, the solution set $\mathcal K_\lambda$ is unbounded if $b\in \partial M_\lambda$, $\mathcal K_\lambda=\{b\}$ if $b\in \mathrm{int}\, M_\lambda$ and $\mathcal K_\lambda=\emptyset$ if $b\not\in M_\lambda$. Hence, $\mathcal J_\lambda$ is coercive if and only if $b\in\mathrm{int}\,M_\lambda$ (see Lemma \ref{lemma_K_t}). We also mention that 
\begin{equation}
	\mathcal I_{\lambda,\omega}(v)=\max_{x\in M_\lambda}\Big\{x^\top v-\frac{1}{2\omega}\|x\|^2\Big\},\quad \mathcal I_{\infty,\lambda}(v)=\sup_{x\in M_\lambda}\{x^\top v\},
	\label{example22}
\end{equation}
as follows from the results presented e.g. in \cite{HRS16}. We consider two special choices of the family $\{M_\lambda\}$.

\begin{example}\label{examp1}
Let $\lambda_0\in(0,1]$, $M_\lambda=\{x\in\mathbb R^n\ |\; \|x\|\leq\lambda^{-1}\}$ for any $\lambda\geq\lambda_0$, $b\neq0$ and $\|b\|<\lambda_0^{-1}$. Then, it is easy to verify that $(\mathcal B_1)_\lambda-(\mathcal B_5)_\lambda$ and $(\mathcal B_3)^+_\lambda$ hold. In addition, $b\in\partial M_\lambda$ for $\lambda=\|b\|^{-1}$ and thus $\lambda^*=\|b\|^{-1}$, $\mathcal K_{\lambda^*}=\{u^*\in\mathbb R^n\ |\; \exists \alpha\geq1:\; u^*=\alpha b\}.$
Hence, $(\mathcal B_8)_\lambda$ hold, but $(\mathcal B_8)^+_\lambda$ does not hold. Assumptions $(\mathcal B_6)_\lambda$ and $(\mathcal B_7)_\lambda$ can be verified, for example, using the following explicit form of $\mathcal I_\lambda$:
$$\mathcal I_\lambda(v)=\left\{\begin{array}{ll}
	\frac{1}{2}\|v\|^2, & \|v\|\leq\frac{1}{\lambda}\\[1mm]
	\frac{1}{\lambda}\|v\|-\frac{1}{2\lambda^2}, & \|v\|\geq\frac{1}{\lambda},
\end{array}\right.
\quad\text{for all } v\in\mathbb R^n.$$
Indeed, it suffices to choose $v_\lambda=\lambda b$ in $(\mathcal B_7)_\lambda$. Therefore, one can apply Theorem \ref{theorem_omega_SSRM} and find the mapping $\psi:\omega\mapsto\lambda$, $\omega\geq\omega_0=\|b\|$, in the following form:
\begin{equation}
    \psi(\omega_0)=[\lambda_0,\lambda^*]\quad\mbox{and}\quad \psi(\omega)=\lambda^*\quad\forall \omega>\omega_0.
    \label{psi_ex}
\end{equation}
Finally, we focus on the verification of the results from Section \ref{subsec_LA4SSRM}. The assumption $(\mathcal B_3)_{\lambda,t}^+$ can be easily checked. $(\mathcal B_9)_\lambda$ and $(\mathcal B_{10})_\lambda$ can be verified, for example, using the following explicit formulas:
$$
\mathcal I_{\omega,\lambda}(v)=\left\{\begin{array}{ll}
\frac{\omega}{2}\|v\|^2, & \|v\|\leq\frac{1}{\omega\lambda}\\[1mm]
\frac{1}{\lambda}\|v\|-\frac{1}{2\omega\lambda^2}, & \|v\|\geq\frac{1}{\omega\lambda},
\end{array}\right.
\quad \mathcal I_{\infty,\lambda}(v)=\frac{1}{\lambda}\|v\|\quad\text{for all } v\in\mathbb R^n,
$$
$\mathcal C_\lambda=\mathbb R^n$ and
$$\ell(\lambda)=\min_{\substack{v\in\mathbb R^2\\ b^\top v=1}}\mathcal I_{\infty,\lambda}(v)=\frac{1}{\lambda\|b\|}\quad\forall \lambda\geq\lambda_0.$$
We see that the function $\ell$ is decreasing and $\ell(\lambda)=1$ for $\lambda=\lambda^*=1/\|b\|$. 
\end{example}

\begin{example}\label{examp2}
Let $\lambda_0\in(0,1]$ and $M_\lambda=\{x=(x_1,x_2)^\top\in\mathbb R^2\ |\; x_1-\lambda|x_2|+1\geq0\}$ for any $\lambda\geq\lambda_0$, see Figure \ref{fig_K}. 
\begin{figure}
	\centering
	\includegraphics[width=0.4\textwidth]{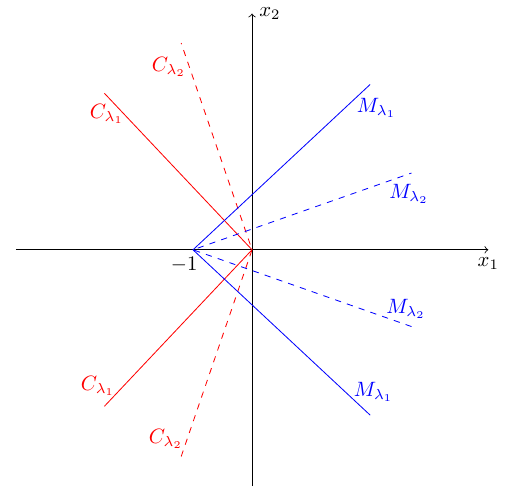}
\caption{Visualization of the sets $M_\lambda$ and $\mathcal C_\lambda$ for $\lambda_1<\lambda_2$.}
\label{fig_K}
\end{figure}
Let $b=(b_1,b_2)^\top$, $b_2\neq0$ and $b_1-\lambda_0|b_2|+1>0$, i.e., $b\in \mathrm{int}\, M_{\lambda_0}$. Then, it is possible to verify  $(\mathcal B_1)_\lambda-(\mathcal B_5)_\lambda$, $(\mathcal B_3)^+_\lambda$ and show that
$$\lambda^*=\frac{1+b_1}{|b_2|}>\lambda_0,\quad \mathcal K_{\lambda^*}=\{b+\alpha w\ |\; \alpha\geq0\},\quad w=\Big(-1,\frac{1+b_1}{b_2}\Big).$$
Further, $(\mathcal B_8)_\lambda$ is satisfied because $b^\top w=1$. However, $(\mathcal B_8)^+_\lambda$ does not hold and $\psi$ is again a multivalued function defined by \eqref{psi_ex}. The remaining assumptions and results are consequences of the following closed forms:
$$\mathcal I_\lambda(v)=\left\{\begin{array}{ll}
	\frac{1}{2}\|v\|^2, & v_1-\lambda|v_2|+1\geq0\\[1mm]
	-v_1-\frac{1}{2}+\frac{(\lambda v_1+|v_2|+\lambda)^2}{2(\lambda^2+1)}, & v_1-\lambda|v_2|+1\leq0,\; v_1+\frac{1}{\lambda}|v_2|+1\geq0,\\[1mm]
	-v_1-\frac{1}{2}, & v_1+\frac{1}{\lambda}|v_2|+1\leq0,
\end{array}\right.$$
$$\mathcal I_{\infty,\lambda}(v)=\left\{\begin{array}{ll}
	-v_1, & v_1+\frac{1}{\lambda}|v_2|\leq0\\[1mm]
	+\infty, & v_1+\frac{1}{\lambda}|v_2|>0,
\end{array}\right.\quad \mathcal C_\lambda=\Big\{v=(v_1,v_2)^\top\in\mathbb R^2\ |\; v_1+\frac{1}{\lambda}|v_2|\leq0\Big\},$$
$$\ell(\lambda)=\frac{1}{\lambda|b_2|-b_1}\quad\text{for all } \lambda\geq\lambda_0,\quad b_2\neq0.$$
We see that $\ell$ is decreasing and $\ell(\lambda^*)=1$ holds.

Let us note that this example is inspired by the Mohr-Coulomb EP constitutive model introduced in the next section. This example is not meaningful for all vectors $b=(b_1,b_2)^\top$. For example, if $b_1<-1$ then $\mathcal K_\lambda=\emptyset$ for any $\lambda>0$ and the key assumption $(\mathcal B_3)_\lambda$ cannot hold. 
\end{example}

\section{Numerical example on slope stability in 3D}
\label{sec_num_example}

In this section, we illustrate the theoretical results for the LL and SSR methods with a numerical example. A static version of the elastic-perfectly plastic model with the Mohr-Coulomb (MC) yield criterion and the associated plastic flow rule is considered. 

First, we briefly describe the corresponding constitutive model, 
which can be defined by the following optimization problem \cite{T85,HR12}:
\begin{equation}
	\Psi(\mbf\varepsilon)=\max_{\mbf\tau,\; f(\mbf\tau)\leq0}\left[\mbf\tau:\mbf\varepsilon-\frac{1}{2}\mathbb D_e^{-1}\mbf\tau:\mbf\tau\right]
	\label{dissip_alpha1}
\end{equation}
Here, $\mbf\varepsilon\in\mathbb R^{3\times3}_{sym}$ is the infinitesimal small strain tensor, $\mbf\tau\in\mathbb R^{3\times3}_{sym}$ is a testing stress tensor, $\mathbb D_e$ denotes the fourth-order elastic tensor representing the Hooke law, and $f$ is the MC yield function defined by
\begin{align}
	f(\mbf\tau)&=(1+\sin\phi)\tau_1-(1-\sin\phi)\tau_3-2c\cos\phi,
\end{align}
where $\tau_1$, $\tau_3$ are the maximal and minimal eigenvalues (principle stresses) of $\mbf\tau$, respectively, $c$ and $\phi$ denote the cohesion and the friction angle, respectively. These material parameters are reduced within the SSR method by the factor $\lambda$. 

Let $\mbf\sigma$ denote the maximizer in \eqref{dissip_alpha1}. Then $\mbf\sigma$ is the Cauchy stress tensor and the implicit function $T\colon\mbf\varepsilon\mapsto\mbf\sigma$ satisfies
\begin{equation}
\frac{\partial\Psi(\mbf\varepsilon)}{\partial\mbf\varepsilon}=T(\mbf\varepsilon)=\mbf\sigma.
\end{equation}
The construction of $T$ and its generalized derivative can be found, e.g., in \cite{NPO08, SCL17}.

Further, the EP problem in terms of displacements have the following weak form:
\begin{equation}
	\mbox{find } {\mbf u}\in V: \quad \int_\Omega T\left(\mbf\varepsilon(\mbf{u})\right):\mbf\varepsilon(\mbf v)\, dx=b(\mbf v) \quad\forall  \mbf v\in V,
	\label{eqn}
\end{equation}
where $\Omega$ is a computational domain, ${V}$ is a Sobolev space of $H^1$-type containing admissible displacements, $\mbf \varepsilon(\mbf u)=\frac{1}{2}\left(\nabla \mbf u+(\nabla \mbf u)^\top\right)$, and the functional $b$ represents a combination of volume and surface external forces. If the problem \eqref{eqn} is discretized by the finite element method, we arrive at the algebraic system \eqref{alg_system}. The corresponding algebraic function $\mathcal I$ can be assembled from the integral
$\int_\Omega \Psi(\varepsilon(\mbf v))\,dx.$

\begin{figure}[h]
	\centering
	\hfill \includegraphics[width=0.5\textwidth]{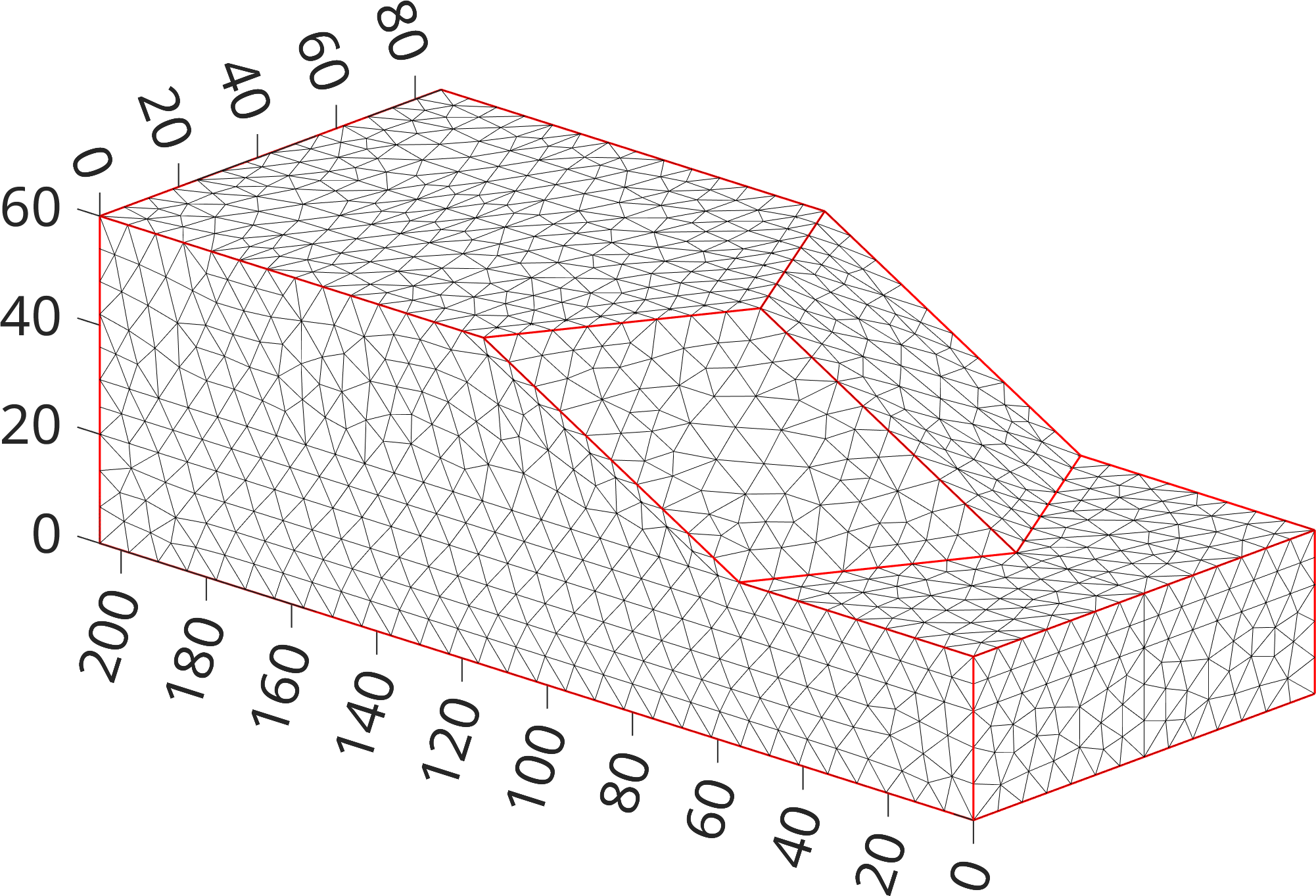}\hfill 
	\caption{Geometry of the computational domain and the investigated mesh.}
	\label{fig_meshes}
\end{figure}

Now, we consider a homogeneous slope depicted in Figure \ref{fig_meshes}. The slope is convex with the angle of $120^\circ$ and the inclination of $31^\circ$. The investigated computational domain is fixed at the bottom and the zero normal displacements are prescribed on all vertical faces. The remaining part of the boundary is free. The gravity load is given by the specific weight $\gamma=20\, $kN/m$^3$. Two elastic material parameters are prescribed: Young's modulus $E=40\,$MPa and Poisson's ratio $\nu=0.3$. Further, we set $\phi=20^{\circ}$ and $c=15\,$kPa. 

For numerical solution, we consider P2 elements with a finite element mesh consisting of 15,356 tetrahedrons and 71,970 unknowns, see Figure \ref{fig_meshes}. Next, the suggested indirect continuation methods for the LL and SSR methods are used. The corresponding extended systems of nonlinear equations are solved by the semismooth Newton method with damping and regularization, see, for example, \cite{CHKS15}. Linearized systems appearing in each Newton's iteration are solved by a deflated variant of Krylov methods as in \cite{KSB25}. In particular, the flexible GMRES method is considered. We use a preconditioner based on separate displacement decomposition \cite{separate_displacement_axelsson} combined with aggregation-based algebraic multigrid (AGMG) \cite{agmg_notay}. This solution concept is implemented using in-house codes in Matlab, which are available for download in \cite{SBBL24}. They build on well-documented codes developed in \cite{SCL17,CSV19,KSB25}.

Details of the continuation curves for the LL and SSR methods are depicted in Figure \ref{fig_continuation}. We see that the curves have the expected properties, i.e., they are increasing and bounded from above by unknown values of $t^*$ and $\lambda^*$, respectively. In particular, we obtain $t^*\doteq2.11$ and $\lambda^*\doteq1.20$ as the maximal values of $t$ and $\lambda$ achieved by these curves. 
\begin{figure}[H]
	\centering
	\hfill \includegraphics[width=0.45\textwidth]{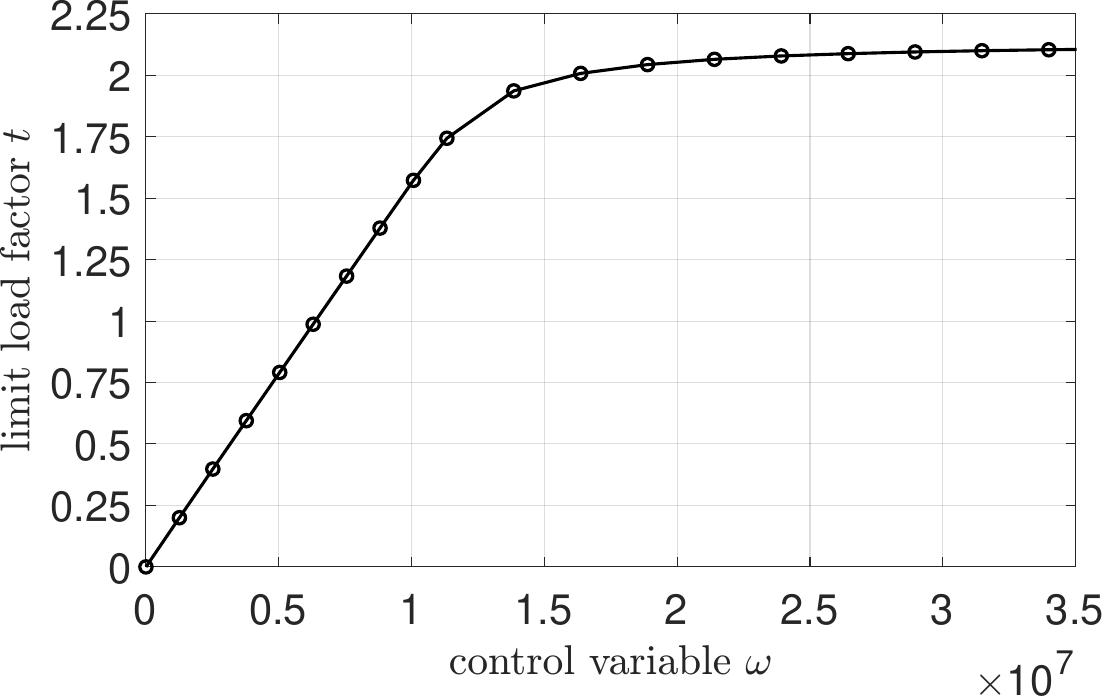}\hfill 
\includegraphics[width=0.45\textwidth]{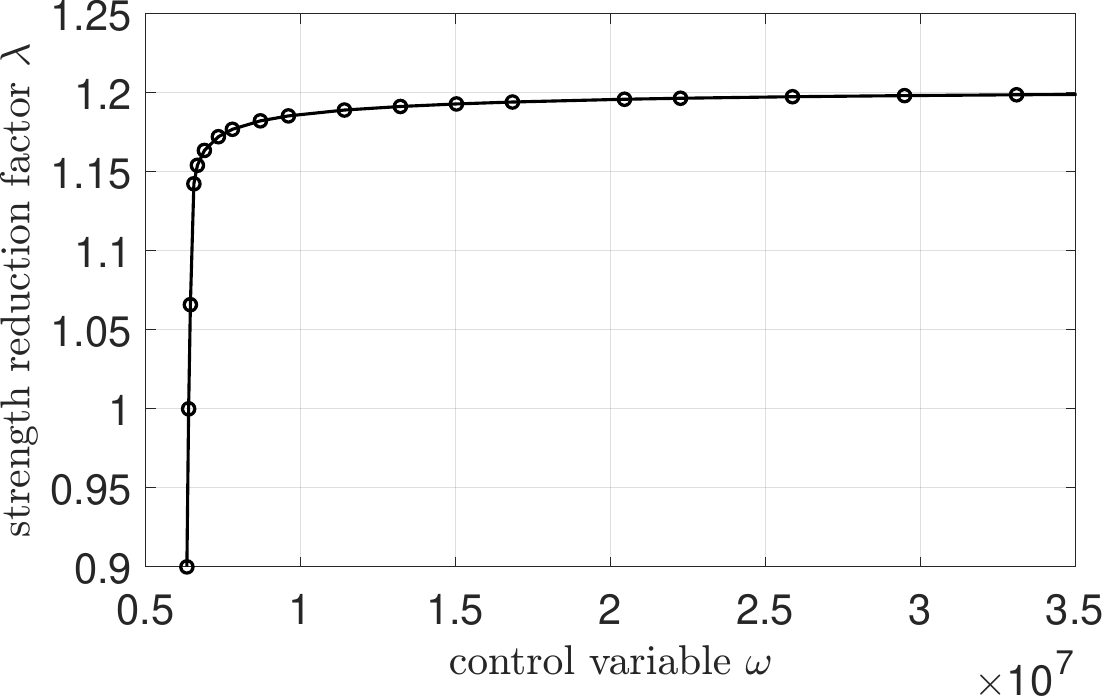}\hfill 
	\caption{Continuation curves for the LL method (left) and the SSR method (right).}
	\label{fig_continuation}
\end{figure}


The corresponding function $\ell$ (introduced in Section \ref{subsec_LA4SSRM}) is depicted in Figure \ref{fig_ell_function}. According to theoretical results, we see that this function is decreasing and its intersection with the value 1 (in red) corresponds to $\lambda^*$. Further, we see that $\ell(1)=t^*\doteq2.11$.
\begin{figure}[H]
	\centering
	\includegraphics[width=0.55\textwidth]{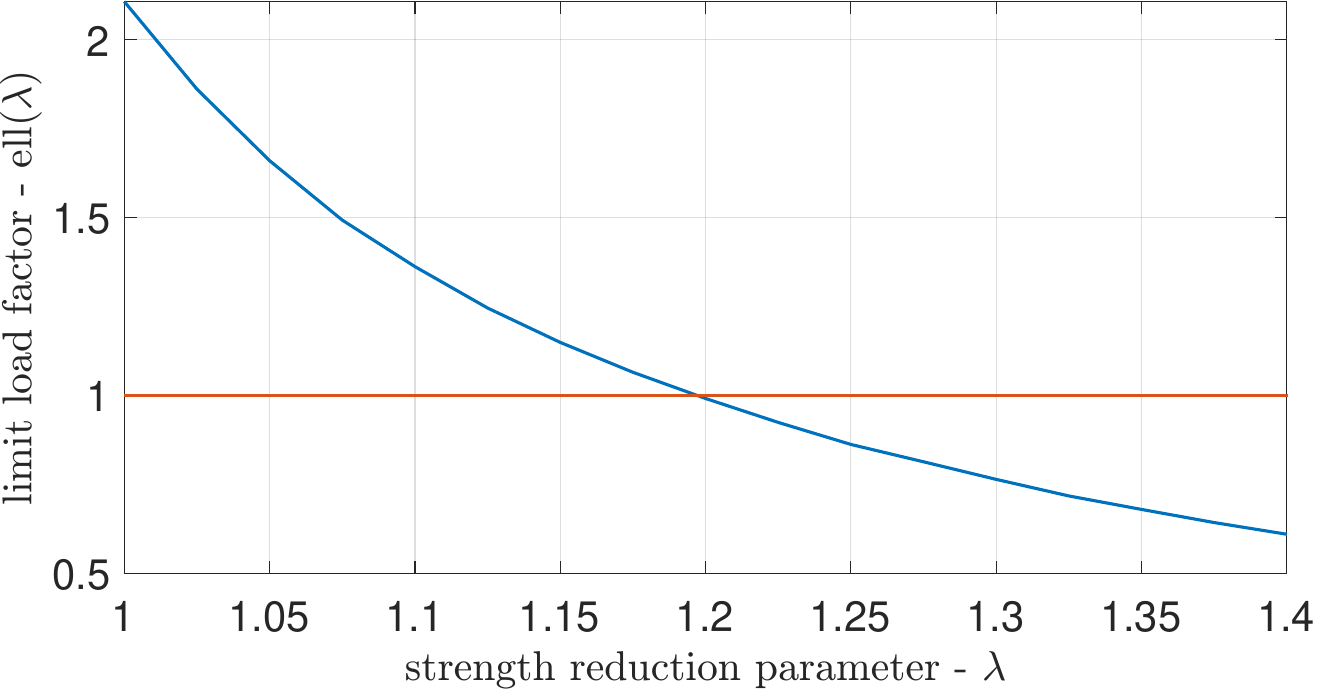}
	\caption{The computed function $\ell$ (in blue).}
	\label{fig_ell_function}
\end{figure}

Finally, one can visualize failure zones corresponding to the safety factors $t^*$ and $\lambda^*$. However, the investigated mesh is too coarse for this purpose. For better visualization, we use a simple mesh adaptivity based on computed values of the norm of the deviatoric strain. For the sake of brevity, the mesh adaptivity is applied only on for the SSR method. The refined mesh consists of 220,510 elements and 917,373 unknowns and the corresponding safety factor $\lambda^*$ is equal to 1.10, thus, it is slightly lower than for the original mesh. The failure zones are depicted in Figure \ref{fig_failure}. For their visualization, the deviatoric strain is used. The mesh refinement is aligned with the failure zones. We see that the failure is relatively deep in some parts of the slope.
\begin{figure}[H]
	\centering
	\includegraphics[width=0.9\columnwidth]{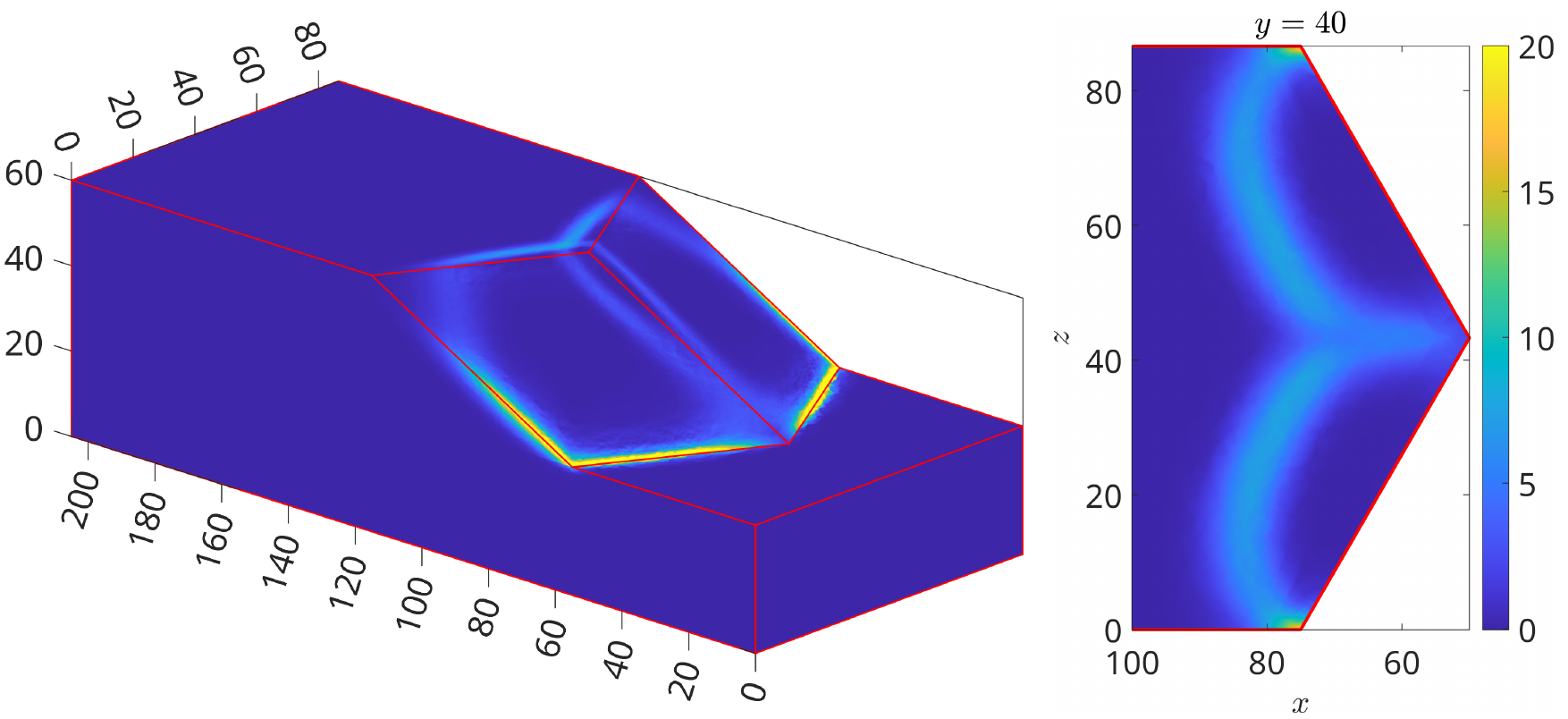}
	\caption{Failure zones for the SSR method computed on a locally refined mesh. Left -- failure on the slope boundary; right -- failure on the horizontal cross section at hight of $40\,$m.}
	\label{fig_failure}
\end{figure}

\section{Conclusion}
\label{sec_conclusion}

In this paper, abstract convex optimization problems in $\mathbb R^n$ inspired by the LL, LA and SSR methods were analyzed and illustrated on numerical and analytical examples. We would highlight the following benefits of the developed abstract framework. 

\begin{enumerate}
    \item To approach the LL, LA and SSR methods known in engineering and geotechnics for mathematicians and other scientists, without the necessity to have knowledge of elasto-plastic models. Especially, the SSR method is practically unknown in mathematical community despite the fact that it can be completed with  mathematical background. 
    \item The algebraic settings of the methods enabled us to easily and transparently introduce advanced continuation  methods. Especially, the suggested indirect continuation method for the SSR method is original to the best of our knowledge. 
    \item We show that the LA problem \eqref{LA_problem} is closely related to solvability analysis in convex optimization, see Theorem \ref{theorem_existence}. The presented analytical examples suggest that this idea could be extended to a broader class of convex functions $\mathcal I$, which are not smooth or do not have a linear growth at infinity. In addition, the LA analysis problem is meaningful even in convex optimization with constraints, see e.g. \cite{SHHC15}.
\end{enumerate}

Finally, it is important to emphasize that there are many other difficulties related to geotechnical stability analysis which are out of the scope of this paper. First, most elasto-plastic models in geotechnics are 'nonassociated' meaning that the investigated optimization framework cannot be done directly and convenient approximations of the models have to be employed. Second, computed FoS can be overestimated and strongly dependent on the mesh density, especially if the simplest linear finite elements are used. Finally, it is also challenging to find convenient iterative solvers for ill-conditioned systems solved in stability analysis.

\medskip\noindent
\textbf{Acknowledgement.} The authors thank Prof. Maya Neytcheva (Uppsala) for valuable comments on this paper.

\bibliographystyle{siamplain}
\bibliography{bibliography}

\end{document}